\documentclass{preprint}

\newcommand{\im}{\operatorname{im}}

\newcommand{\cov}{\operatorname{cov}}

\newcommand{\topmap}{\operatorname{top}}
\newcommand{\botmap}{\operatorname{bot}}

\newcommand{\cB}{\mathcal{B}}
\newcommand{\cC}{\mathcal{C}}
\newcommand{\cF}{\mathcal{F}}
\newcommand{\cP}{\mathcal{P}}
\newcommand{\cR}{\mathcal{R}}
\newcommand{\cT}{\mathcal{T}}

\usepackage{tcolorbox}

\usepackage{graphicx}
\usepackage{tikz}
\usepackage{wrapfig}
\usetikzlibrary{patterns}

\usepackage{bm} 

\makeatletter
\renewcommand{\tocsection}[3]{%
\indentlabel{\@ifnotempty{#2}{\ignorespaces#1 #2\quad}}#3\dotfill}
\makeatother

\usepackage[foot]{amsaddr}
\usepackage[hang, flushmargin]{footmisc}

\theoremstyle{definition}
\newtheorem*{notation*}{Notation}

\usepackage{lipsum}

\counterwithin*{equation}{section}

\title[Boolean Antichains]{Counting Boolean antichains}

\author{Julien Garber}
\address[J.G., L.S.G, T.G.]{Fakult{\"a}t f{\"u}r Mathematik, Ruhr-Universit{\"a}t Bochum, Universit{\"a}tstr. 150, 44801 Bochum}
\email{julien.garber@rub.de}
\author{Lina S. Goltermann}
\email{lina.goltermann@edu.rub.de}
\author{Daniel Horiatakis}
\email{daniel.horiatakis@uni-graz.at}
\address[D.H.]{Institut f{\"u}r Mathematik und Wissenschaftliches Rechnen, Universität Graz, Heinrichstraße 36, 8010 Graz, Austria}

\author{Dennis König}
\address[D.K.]{Universit{\"a}t zu K{\"o}ln, Department Mathematik/Informatik, Abteilung Mathematik, Weyertal 86-90, 50931 K{\"o}ln, Germany}

\author{Tal Gottesman}
\email{tal.gottesman@rub.de}

\keywords{antichains, lattice congruence, Boolean lattice, partition lattice, Tamari lattice}
\usepackage[style=alphabetic, maxnames=99, maxalphanames=99]{biblatex}
\DeclareLabelalphaTemplate{
  \labelelement{
    \field[final]{shorthand}
    \field{label}
    \field[strwidth=2,strside=left,ifnames=1]{labelname}
    \field[strwidth=1,strside=left]{labelname}
  }
}
\DeclareFieldFormat{extraalpha}{#1}
\begin{document}

\begin{abstract}
    We say that an antichain in a lattice $L$ is Boolean if it generates a Boolean sublattice in $L$ in a particularly nice way. These purely combinatorial objects play a role in the representation theory of the incidence algebra of the lattice $L$, as indicated by recent results of Rognerud, Yıldırım and of the last author with Klász, Kleinau and Marczinzik. Given these motivations, it is natural to ask: Can we count and construct Boolean antichains in certain lattices? In this paper we give a construction of Boolean antichains in Boolean lattices, partition lattices and Tamari lattices. Furthermore, we share the surprisingly elegant formulas we found for the number of these Boolean antichains.
\end{abstract}

\maketitle
\tableofcontents

\section{Introduction} 
Antichains appear in the representation theory of the incidence algebra $A$ of a finite lattice $L$. They describe order ideals, or lower sets, of a lattice by the set of their maxima, leading to a parametrisation of modules over $A$ that have a simple top \cite[Proposition 2.1]{iyama2022distributive}.
This is somewhat surprising in module categories that are most often wild \cite{loupias1975indecomposable}. Within these well behaved modules, those associated to \emph{Boolean antichains} are so-called perfect modules \cite{GKKM}. More crucially, Boolean antichains play an important role in results about fractionally Calabi-Yau posets (\cite{gottesman2025fractionally},\cite{rognerud2021bounded}, \cite{yildirim2019coxeter}, \cite{kleinau2026cambrian}), which fit in a far-reaching conjecture of Chapoton \cite{chapoton:hal-04037012}. 
The fact that they appear in these concrete examples raises the question: Is this a coincidence? Or do there simply exist many Boolean antichains in general? This leads to the better-suited question: Can we count and construct them, particularly in non-distributive settings?

Counting antichains in general is hard. For instance, Dedekind's problem \cite{dedekind1897decomposition} asks: How many antichains are there in a Boolean lattice? To this day, there is no closed-form formula. Only bounds \cite{kleitman1969dedekind} and the exact values up to $n = 9$ \cite[Entry A000372]{oeis2025dedekind} are known. However, when we restrict to Boolean antichains, the problem becomes tractable, as we are able to count Boolean antichains not only in Boolean lattices but also in several other families of lattices. The paper is organised as follows.

In \Cref{sec: Generalities}, we define the notion of Boolean antichains and give bounds on their size.
We use our first results to count Boolean antichains in the Boolean lattice in \Cref{sec: Boolean lattices}.
\begin{theorem*}[\Cref{thm:main_thm_for_Boolean_lattices}]
    The number of Boolean antichains of size $k$ in the Boolean lattice $\cB_n$ is given by the Stirling number of the second kind $S(n+1,k+1)$ \cite{stirling1749differential}.
\end{theorem*}
In \Cref{sec: Partition lattices}, we give a bijection from Boolean antichains of maximal size in the partition lattice to labelled trees. Thus, both are counted by Cayley's formula \cite{Cayley1897trees}.
\begin{theorem*}[\Cref{thm: main_thm_partition_lattice}]
    The number of Boolean antichains of size $n-1$ in the partition lattice $\cP_n$ is $n^{n-2}$.
\end{theorem*}
By restricting this bijection to non-crossing labelled trees, we count Boolean antichains of maximal size in the lattice of non-crossing partitions.

\Cref{sec: Tamari lattices} is devoted to the Tamari lattice and constitutes the core effort of this work. 
In \Cref{sec: lattice congruence} we introduce a lattice congruence $q:\cT_{n+1}\to \cT_{n}$. After establishing the technical results in \Cref{sec: closing index sets}, we use the lattice congruence to construct the Boolean antichains of $\cT_{n+1}$ in a recursive way (\Cref{sec: recursive construction}). In turn, this allows us to count Boolean antichains of maximal size in the Tamari lattice (\Cref{sec: Maximal size}), yielding our main result.
\begin{theorem*}[\Cref{thm: main_thm_tamari_lattice}]
    The number of Boolean antichains of size $n$ in the Tamari lattice $\cT_{n+1}$ is given by the Catalan number $\mathtt{c}_n$.
\end{theorem*}
Our results lead us to formulate further questions regarding Boolean antichains in \Cref{sec: Perspectives}, which, although quite simple, have not yet been thoroughly studied.

\begin{notation*}
We denote the least (resp.\! greatest) element of $L$ by $\hat{0}_L=\hat{0}$ (resp.\! $\hat{1}_L=\hat{1}$). For $x$ and $y$ elements of $L$, the inequality $x< y$ is a \emph{covering relation} if $x\leq c\leq y$ implies $c = x$ or $c = y$. In that case we write $x\lessdot y$ and we say that $y$ \emph{covers} $x$. We write $\cov(y) = \{z\in L | z\lessdot y\}$ for the set of elements covered by $y$. We call elements that cover $\hat{0}$ \emph{atoms}, and elements that are covered by $\hat{1}$ \emph{coatoms}. The set of integers $\{1, \dots, n\}$ is denoted by $[n]$, and more generally, $[m, n]$ denotes $\{m, \dots, n\}$ whenever $m\leq n$ are integers. The power set of $C$ is denoted by $\mathscr{P}(C)$. 
\end{notation*}

\paragraph{\textbf{Acknowledgements}} This project is the result of a ``Dive into Research'', organized by Christian Stump at Ruhr University Bochum and supported by the DFG SPP 2458 ``Combinatorial Synergies''. In particular, the authors thank Karin Baur for supervising the project at the ``Dive into Research''. The topic was given by the fifth author. The first four authors developed their own proofs and wrote this paper under supervision of the fifth author and Karin Baur. The authors thank Benjamin Hackl for several comments on the paper, and in particular for drawing attention to \Cref{rem: genPartition}. The third author was partially funded by the Deutsche Forschungsgemeinschaft (DFG, German Research Foundation) -- Projektnummer 496500943.

\section{Generalities}\label{sec: Generalities}
Let $L$ be a finite lattice whose meet and join operations are denoted by $\wedge$ and $\vee$ respectively. A subset $C$ of $L$ is an antichain if it consists of pairwise incomparable elements of $L$. Following \cite{gottesman2025fractionally}, we say that an antichain $C$ is \emph{Boolean} if it does not contain the greatest element of $L$ and, for all subsets $S$ and $S'$ of $C$, we have 
\begin{equation}\label{eq: intersectivity}
    (\wedge S) \vee (\wedge S') = \wedge (S\cap S').
\end{equation}
An example of a Boolean antichain can be seen in \Cref{fig:Boolean} in \Cref{sec: Boolean lattices}.
Boolean antichains owe their name to the following lemma since the power set of $C$ is a Boolean lattice. For this, let $(\langle C, \hat{1}\rangle_{\vee, \wedge},  \wedge, \vee)$ denote the smallest sublattice of $L$ containing $C$ and $\hat{1}$. We call it the \emph{span} of $C$.
\begin{lemma}[\protect{\cite[Lemma 2.1.5]{gottesman2025fractionally}}]\label{lem:BoolBool}
An antichain is Boolean if and only if the map 
\begin{align*}
(\mathscr{P}(C), \cap, \cup)&\xrightarrow{\phi} (\langle C, \hat{1}\rangle_{\vee, \wedge},  \wedge, \vee)\\
S\quad\quad &\mapsto \quad\quad \wedge S
\end{align*}
is a lattice anti-isomorphism.
\end{lemma}

\begin{remark}\label{rem:Intersective is antichain}
    Let $M$ be a subset of $L\setminus\{\hat{1}\}$ which satisfies \eqref{eq: intersectivity}. Then $M$ is an antichain and hence by definition a Boolean antichain.
    To see this, we assume that $M$ contains comparable elements $a\leq b$. Choosing $S=\{a\}$ and $S'=\{b\}$ yields
    $$
    (\wedge S)\vee(\wedge S')=b < \hat{1} = \wedge (S\cap S'),
    $$
    contradicting $M$ fulfilling \eqref{eq: intersectivity}.
\end{remark}

\begin{remark}\label{lem:single elements join to top}
    Let $c\neq c'$ be elements of a Boolean antichain $C$, then $c\vee c' = \hat{1}$. This follows from \eqref{eq: intersectivity} by choosing $S=\{c\}$ and $S'=\{c'\}$. 
\end{remark}

For a distributive lattice, the converse of \Cref{lem:single elements join to top} holds as well, thereby simplifying the condition for an antichain to be Boolean.
\begin{lemma}\label{lem:distributive}
    Let $L$ be a distributive lattice, and let $C$ be a subset of $L\setminus\{\hat{1}\}$. Then $C$ is a Boolean antichain if and only if $c\vee c' = \hat{1}$ for all $c\neq c' \in C$.
\end{lemma}
\begin{proof}
    We only need to show that the condition is sufficient.
    Assume that $c\vee c'=\hat{1}$ for all $c\neq c'\in C$.
    By using the distributivity of $L$ and $c\vee c = c$, we obtain
    \[
        \left(\bigwedge S\right)\vee \left(\bigwedge S'\right) = \left( \bigwedge_{c \in S} c\right) \vee \left( \bigwedge_{c' \in S'} c' \right) = \bigwedge_{\genfrac{}{}{0pt}{}{c \in S}{c' \in S'}} ( c \vee c' ) \: = \bigwedge_{c\in S\cap S'}c \: =\bigwedge(S\cap S').
    \]
    This shows \eqref{eq: intersectivity}. By \Cref{rem:Intersective is antichain}, $C$ is a Boolean antichain.
\end{proof}

We conclude this section by providing bounds on the size of a Boolean antichain in terms of various combinatorial statistics related to the lattice $L$. For this, we recall that the \textit{height} $h(L)$ of a lattice $L$ is defined as the maximum number of consecutive cover relations in $L$.

\begin{lemma}\label{lem:size bounds}
    If $C$ is a Boolean antichain in $L$, then 
    \begin{enumerate}
        \item\label{item:size bounds cov} $|C| \leq | \cov(\hat{1})|$ and
        \item\label{item:size bounds height} $|C| \leq h(L)$.
    \end{enumerate}
\end{lemma}
\begin{proof}
    Inequality \eqref{item:size bounds cov} follows from \Cref{lem:single elements join to top} and the pigeon hole principle.
    Moreover, \Cref{lem:BoolBool} yields that $C$ generates a Boolean sublattice of height $|C|$. This shows inequality \eqref{item:size bounds height}.
\end{proof}
From here on, we write $\cC_k(L)$ to denote the set of Boolean antichains of size $k$ in a lattice $L$. The rest of the paper is devoted to describing this set for various lattices.

\section{Boolean lattices}\label{sec: Boolean lattices}
The Boolean lattice $\cB_n$ consists of all subsets of $[n]$, ordered by inclusion. Recall that the join of two subsets in a Boolean lattice is given by their union. As a first result, we use \Cref{lem:distributive} to count Boolean antichains in Boolean lattices. 
For this, we denote the set of partitions of $[n+1]$ into $k+1$ subsets by $\cP_{n+1,k+1}$.

\begin{theorem}\label{thm:main_thm_for_Boolean_lattices}
    The number of Boolean antichains of size $k$ in the Boolean lattice $\cB_n$ is given by the Stirling number of the second kind $S(n+1,k+1)$.
\end{theorem}
\begin{proof}
    We construct a bijection between the set $\cC_k(\cB_n)$ and the set $\cP_{n+1,k+1}$. For this, we number the blocks of a partition $\varrho = \{\varrho_1, \dots, \varrho_{k+1}\}$ such that the element $n+1$ is contained in the block $\varrho_{k+1}$. 
    For $i\neq j \in \{1, \dots, k\}$, we compute
    \[
    ([n]\setminus \varrho_i) \cup ([n] \setminus \varrho_j) = [n] = \hat{1}_{\cB_n}.
    \]
    Using \Cref{lem:distributive}, this shows that the map
    \begin{align*}
        \Phi_{n,k}: \cP_{n+1,k+1} &\rightarrow \cC_k(\cB_n) \\
        \varrho &\mapsto \{[n] \setminus \varrho_1, \dots, [n] \setminus \varrho_k \}
    \end{align*}
    is well-defined. On the other hand, let $C=\{c_1,\dots,c_k\}$ be a Boolean antichain in $\cB_n$. By \Cref{lem:single elements join to top}, the equality $c_i \cup c_j = [n]$ holds. This yields
    \[
        ([n] \setminus c_i) \cap ([n] \setminus c_j) = [n] \setminus (c_i \cup c_j) = \varnothing.
    \]
    Furthermore, we have
    \[
        [n+1]\setminus \Bigl( \bigcup_{\ell=1}^{k} ([n] \setminus c_\ell)\Bigr) = \Bigl( \bigcap_{\ell=1}^k c_\ell\Bigr) \cup \{n+1\}.
    \]
    This shows that the map
    \begin{align*}
    \Psi_{n,k}:\quad \cC_k(\cB_n) \quad &\rightarrow \qquad \cP_{n+1,k+1} \\
    \{c_1, \dots, c_k\} \; &\mapsto \; \left\{ [n] \setminus c_1, \dots, [n] \setminus c_k, \left(\bigcap_{\ell=1}^{k} c_\ell\right) \cup \{n+1\}\right\}
    \end{align*}
    is well-defined. A computation shows that $\Phi_{n,k}$ and $\Psi_{n,k}$ are inverses of each other. We conclude by using that partitions of $[n+1]$ into $k+1$ subsets are counted by the Stirling number of the second kind $S(n+1,k+1)$ (\cite{stirling1749differential} or \cite[Section 6.1]{Stirling_numbers} for an exposition).
\end{proof}

\begin{corollary}\label{Total number of Boolean antichains is Bell number}
    The total number of Boolean antichains in the Boolean lattice $\cB_n$ is given by the Bell number \cite{bell1934exponential}
    $$B_{n+1} := \sum_{i=1}^{n+1} S(n+1,i).$$
\end{corollary}

\begin{example}
    Consider the Boolean lattice $\cB_4$ and the antichain $C=\{\{2,4\},\{1,3,4\},\{1,2,3\}\}$ as shown on the left-hand side of \Cref{fig:Boolean}. It generates the sublattice colored in gray which is isomorphic to $\cB_3$. By \Cref{lem:BoolBool}, $C$ is a Boolean antichain. Its corresponding partition is $\Psi_{4,3}(C)=\{\{1,3\},\{2\},\{4\},\{5\}\}$. 
    
    On the other hand, the antichain $C'=\{\{2,4\},\{3,4\},\{1,2,3\}\}$, as shown on the right-hand side of \Cref{fig:Boolean}, is not a Boolean antichain. This can be verified by the fact that $\{2,4\}\vee\{3,4\} = \{2,3,4\}$ is not the greatest element (see \Cref{lem:single elements join to top}).
\end{example}

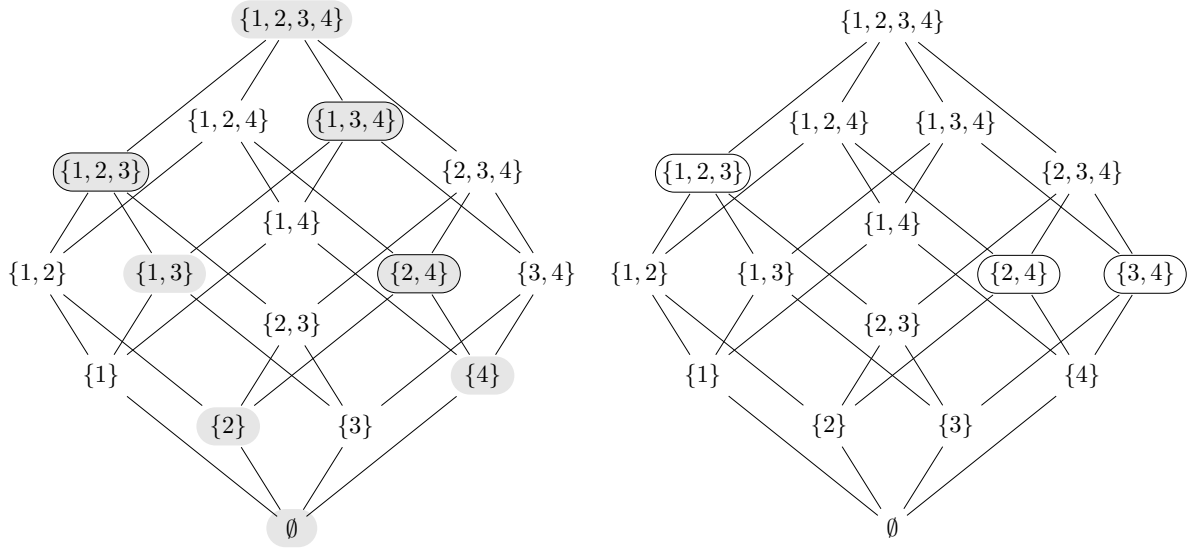
\begin{figure}
    \centering
    \begin{minipage}{0.49\textwidth}
    \resizebox{1\textwidth}{!}{
    \begin{tikzpicture}
        \fill[rounded corners = 3mm, fill=llgray] (0.25, 5.3) rectangle (1.75, 5.9);
        \fill[rounded corners = 3mm, fill=llgray] (4.25, 6.1) rectangle (5.75, 6.7);
        \fill[rounded corners = 3mm, fill=llgray] (5.35, 3.7) rectangle (6.65, 4.3);
        \fill[rounded corners = 3mm, fill=llgray] (6.5, 2.1) rectangle (7.5, 2.7);
        \fill[rounded corners = 3mm, fill=llgray] (2.5, 1.3) rectangle (3.5, 1.9);
        \fill[rounded corners = 3mm, fill=llgray] (1.35, 3.7) rectangle (2.65, 4.3);
        \fill[rounded corners = 3mm, fill=llgray] (3.05, 7.7) rectangle (4.95, 8.3);
        \fill[rounded corners = 3mm, fill=llgray] (3.6, -0.3) rectangle (4.4, 0.3);
        \node (0) at (4, 0) {$\emptyset$};
        \node (1) at (1, 2.4) {$\{1\}$};
        \node (2) at (3, 1.6) {$\{2\}$};
        \node (3) at (5, 1.6) {$\{3\}$};
        \node (4) at (7, 2.4) {$\{4\}$};
        \node (5) at (0, 4) {$\{1,2\}$};
        \node (6) at (2, 4) {$\{1,3\}$};
        \node (7) at (4, 4.8) {$\{1,4\}$};
        \node (8) at (4, 3.2) {$\{2,3\}$};
        \node (9) at (6, 4) {$\{2,4\}$};
        \node (10) at (8, 4) {$\{3,4\}$};
        \node (11) at (1, 5.6) {$\{1,2,3\}$};
        \node (12) at (3, 6.4) {$\{1,2,4\}$};
        \node (13) at (5, 6.4) {$\{1,3,4\}$};
        \node (14) at (7, 5.6) {$\{2,3,4\}$};
        \node (15) at (4, 8) {$\{1,2,3,4\}$};
        
        \draw (0) -- (1);
        \draw (0) -- (2);
        \draw (0) -- (3);
        \draw (0) -- (4);
        \draw (1) -- (5);
        \draw (1) -- (6);
        \draw (1) -- (7);
        \draw (2) -- (5);
        \draw (2) -- (8);
        \draw (2) -- (9);
        \draw (3) -- (6);
        \draw (3) -- (8);
        \draw (3) -- (10);
        \draw (4) -- (7);
        \draw (4) -- (9);
        \draw (4) -- (10);
        \draw (5) -- (11);
        \draw (5) -- (12);
        \draw (6) -- (11);
        \draw (6) -- (13);
        \draw (7) -- (12);
        \draw (7) -- (13);
        \draw (8) -- (11);
        \draw (8) -- (14);
        \draw (9) -- (12);
        \draw (9) -- (14);
        \draw (10) -- (13);
        \draw (10) -- (14);
        \draw (11) -- (15);
        \draw (12) -- (15);
        \draw (13) -- (15);
        \draw (14) -- (15);
        \draw[rounded corners = 3mm] (0.25, 5.3) rectangle (1.75, 5.9);
        \draw[rounded corners = 3mm] (4.25, 6.1) rectangle (5.75, 6.7);
        \draw[rounded corners = 3mm] (5.35, 3.7) rectangle (6.65, 4.3);
    \end{tikzpicture}}
    \end{minipage}
    \begin{minipage}{0.49\textwidth}
    \resizebox{1\textwidth}{!}{
    \begin{tikzpicture}
        \node (0) at (4, 0) {$\emptyset$};
        \node (1) at (1, 2.4) {$\{1\}$};
        \node (2) at (3, 1.6) {$\{2\}$};
        \node (3) at (5, 1.6) {$\{3\}$};
        \node (4) at (7, 2.4) {$\{4\}$};
        \node (5) at (0, 4) {$\{1,2\}$};
        \node (6) at (2, 4) {$\{1,3\}$};
        \node (7) at (4, 4.8) {$\{1,4\}$};
        \node (8) at (4, 3.2) {$\{2,3\}$};
        \node (9) at (6, 4) {$\{2,4\}$};
        \node (10) at (8, 4) {$\{3,4\}$};
        \node (11) at (1, 5.6) {$\{1,2,3\}$};
        \node (12) at (3, 6.4) {$\{1,2,4\}$};
        \node (13) at (5, 6.4) {$\{1,3,4\}$};
        \node (14) at (7, 5.6) {$\{2,3,4\}$};
        \node (15) at (4, 8) {$\{1,2,3,4\}$};
        
        \draw (0) -- (1);
        \draw (0) -- (2);
        \draw (0) -- (3);
        \draw (0) -- (4);
        \draw (1) -- (5);
        \draw (1) -- (6);
        \draw (1) -- (7);
        \draw (2) -- (5);
        \draw (2) -- (8);
        \draw (2) -- (9);
        \draw (3) -- (6);
        \draw (3) -- (8);
        \draw (3) -- (10);
        \draw (4) -- (7);
        \draw (4) -- (9);
        \draw (4) -- (10);
        \draw (5) -- (11);
        \draw (5) -- (12);
        \draw (6) -- (11);
        \draw (6) -- (13);
        \draw (7) -- (12);
        \draw (7) -- (13);
        \draw (8) -- (11);
        \draw (8) -- (14);
        \draw (9) -- (12);
        \draw (9) -- (14);
        \draw (10) -- (13);
        \draw (10) -- (14);
        \draw (11) -- (15);
        \draw (12) -- (15);
        \draw (13) -- (15);
        \draw (14) -- (15);
        \draw[rounded corners = 3mm] (0.25, 5.3) rectangle (1.75, 5.9);
        \draw[rounded corners = 3mm] (5.35, 3.7) rectangle (6.65, 4.3);
        \draw[rounded corners = 3mm] (7.35, 3.7) rectangle (8.65, 4.3);
    \end{tikzpicture}}
    \end{minipage}
    \caption{\centering Examples of a Boolean antichain and a Non-Boolean antichain in $\cB_4$}
    \label{fig:Boolean}
    \end{figure}

\begin{remark}\label{rem: boolean lattice recursion}
    Stirling numbers of the second kind satisfy the recursive formula \cite[Equation (6.3)]{Stirling_numbers}
    $$
        S(n+1,k+1)=S(n,k)+(k+1)S(n,k+1).
    $$
    One can check that $\cC_k(\cB_n)$ is in bijection with the set
    \begin{align*}
        \{\{d_1\cup\{n\}, \dots, d_{k-1} \cup \{n\}, [n-1]\}&\colon \{d_1,\dots,d_{k-1}\} \in \cC_{k-1}(\cB_{n-1})\} \;\cup \\ \{\{d_1\cup\{n\},\dots,d_k\cup \{n\}\}&\colon \{d_1,\dots,d_k\} \in \cC_k(\cB_{n-1})\}\;\cup \\ \{\{d_1\cup\{n\}, \dots,d_{i-1}\cup \{n\},d_i, d_{i+1}\cup \{n\},\dots, d_k \cup \{n\}\}&\colon i\in[k] \text{ and }\{d_1,\dots,d_k\} \in \cC_k(\cB_{n-1})\}.
    \end{align*}
    Hence, the quantity $|\cC_k(\cB_n)|$ satisfies the recursive formula as well. This yields an alternative proof of \Cref{thm:main_thm_for_Boolean_lattices}, without giving a bijection to $\cP_{n+1,k+1}$.

\end{remark}

\begin{remark}
    Let $n$ be a positive integer and let $D_n$ be the set of positive divisors of $n$. Ordered by divisibility, $D_n$ is a distributive lattice.
    Using the prime decomposition $n = p_1^{\ell_1}\cdots p_r^{\ell_r}$, we can generalize \Cref{thm:main_thm_for_Boolean_lattices} as follows \cite{garber2026combinatorial}:
    \[
        |\cC_k(D_n)| = \sum_{m=k}^r S(m,k)\cdot \sum_{1\leq i_1<\dots<i_m\leq r} \ell_{i_1}\cdots \ell_{i_m}. 
    \]
    Note that setting $\ell_1 = \dots = \ell_r = 1$ recovers $D_n = \cB_r$. Using a well-known identity for Stirling numbers of the second kind \cite[Equation (6.15)]{Stirling_numbers}, we also recover \Cref{thm:main_thm_for_Boolean_lattices} since
    \[
        |\cC_k(\cB_r)| = \sum_{m=k}^r S(m,k)\cdot \sum_{1\leq i_1<\dots<i_m\leq r} 1 = \sum_{m=k}^r S(m,k)\cdot \binom{r}{m}=S(r+1,k+1).
    \]
\end{remark}

\section{Partition lattices}\label{sec: Partition lattices}
We saw in \Cref{thm:main_thm_for_Boolean_lattices} that partitions of the set $[n+1]$ are in bijection with Boolean antichains of the Boolean lattice $\cB_n$. In this section, we consider a lattice structure on partitions. Recall that a partition $\varrho'$ \textit{refines} a partition $\varrho$ if every block of $\varrho$ is partitioned by a set of blocks in $\varrho'$. Ordered by refinement, the set of partitions of $[n]$, denoted by $\cP_n$, forms a lattice known as the \textit{partition lattice}.

We determine the number of Boolean antichains of size $n-1$ in $\cP_n$ by establishing a bijection between $C_{n-1}(\cP_n)$ and the set of \textit{labelled trees} with $n$ vertices, denoted by $\mathrm{Tr}_n$. This set consists of all undirected, connected, cycle-free graphs in which each vertex is labelled with a distinct integer from $1$ to $n$. An example of a labelled tree can be seen in \Cref{exm: partition lattice}.

Let $C$ be a Boolean antichain of size $n-1$ in $\cP_n$. By \Cref{lem:BoolBool}, the span of $C$, which we denote by $\text{span}(C)$, is a sublattice isomorphic to $\cB_{n-1}$. To associate a labelled tree with the antichain $C$, we consider the set
$$
A_C:=\{a \in \text{span}(C) \;|\; \hat{0}_{\cP_n}\lessdot a\}.
$$
This set consists of the atoms of the span of $C$ which are also atoms of the ambient lattice $\cP_n$. As such, a partition $a \in A_C$ consists of $n-2$ blocks of size $1$ and one block $e_a$ of size $2$.

\begin{lemma}\label{lem: psi}
Let $C$ be a Boolean antichain in $\cP_n$ of size $n-1$. Then the graph \[T := ([n], \{e_a | a\in A_C\})\] is a labelled tree.
In particular, the map
\begin{align*}
    \psi:\cC_{n-1}(\cP_n)&\rightarrow \mathrm{Tr}_n\\
    C &\mapsto ([n],\{e_a \; | \; a \in A_C\})
\end{align*}
is well-defined.
\end{lemma}

\begin{proof}
Consider the partition $\varrho \in \cP_n$ given by the vertex labels of the connected components of $T$. Since each $e_a$ describes an edge in $T$, we have $a \leq \varrho$ for all $a\in A_C$. In particular, this gives $\varrho \geq \vee A_C$. Since the span of $C$ is a Boolean lattice, $\vee A_C = \hat{1}_{\text{span}(C)}$ holds. Using \Cref{lem:single elements join to top} yields $\hat{1}_{\text{span}(C)}=\hat{1}_{\cP_n}=\{[n]\}$. Overall, this gives $\varrho=\{[n]\}$, which shows that $T$ is connected. Since $T$ has $n-1$ edges and $n$ vertices, $T$ must be a tree.
\end{proof}

Conversely, an antichain can be constructed from a given labelled tree $T=(V(T),E(T))\in \mathrm{Tr}_n$. For this, we define
$$
    \cF(T):=\{F=(V(T),E_F)\; | \; E_F\subseteq E(T)\}
$$
as the set of sub-forests of $T$. Ordered by edge-inclusion, $\cF(T)$ forms a lattice isomorphic to $\cB_{n-1}$.

\begin{lemma}\label{lem: pi properties}
    Let $T \in \mathrm{Tr}_n$ and \begin{align*}
    \pi_T:\cF(T)&\rightarrow \cP_n\\
    F &\mapsto \{V_1,\dots,V_k\},
    \end{align*}
    where $V_1,\dots,V_k$ denote the sets of vertex labels of the connected components of $F$. Then $\pi_T$ is an injective lattice morphism.
\end{lemma}
\begin{proof}
    The injectivity of $\pi_T$ follows from the fact that every connected subgraph of a tree is uniquely determined by its vertex labels. Let $F$ and $F'$ be sub-forests of $T$. We now show that
    $$
    \pi_T(F\wedge F') = \pi_T(F)\wedge \pi_T(F')
    $$
    holds. To this end, we observe that $x, y\in [n]$ lie in the same block of $\pi_T(F)\wedge \pi_T(F')$ if and only if they lie in the same block of $\pi_T(F)$ and in the same block of $\pi_T(F')$. This is the case if and only if the vertices $x$ and $y$ are connected by a path both in $F$ and in $F'$. Because $F$ and $F'$ are sub-forests of the same tree $T$, this must be the unique path connecting $x$ and $y$. Hence, the above statement is equivalent to $x$ and $y$ lying in the same block of $\pi_T(([n],E_F\cap E_{F'}))$. Since $\cF(T)$ forms a Boolean lattice, this is $\pi_T(F\wedge F')$. The argument for $\pi_T(F\vee F') = \pi_T(F)\vee \pi_T(F')$ proceeds analogously and is therefore left to the reader.
\end{proof}

\begin{lemma}\label{lem: phi}
Let $T\in \mathrm{Tr}_n$ be a labelled tree. Then the set $\cov(\hat{1})\cap \im(\pi_T)$ is a Boolean antichain of size $n-1$ in $\cP_n$. In particular, the map 
\begin{align*}
    \varphi:\mathrm{Tr}_n&\rightarrow \cC_{n-1}(\cP_n)\\
    T &\mapsto \cov(\hat{1})\cap{\im(\pi_T)}
\end{align*}
is well-defined.
\end{lemma}

\begin{proof}
    By \Cref{lem: pi properties}, $\im(\pi_T)$ is a Boolean lattice of height $n-1$. Since the ambient lattice $\cP_n$ also has height $n-1$, the coatoms of $\im(\pi_T)$ are coatoms of $\cP_n$. This shows that $\im(\pi_T)$ is the span of $\cov(\hat{1})\cap{\im(\pi_T)}$.
    Thus, by \Cref{lem:BoolBool}, the set $\cov(\hat{1})\cap{\im(\pi_T)}$ is a Boolean antichain of size $n-1$ in $\cP_n$.
\end{proof}

We are now able to prove the main result of this section.

\begin{theorem}\label{thm: main_thm_partition_lattice}
    The number of Boolean antichains of size $n-1$ in the partition lattice $\cP_n$ is $n^{n-2}$.
\end{theorem}
\begin{proof}
Let $C$ be in $\cC_{n-1}(\cP_n)$. We notice that the atoms of $\im(\pi_{\psi(C)})$ are precisely the atoms of the span of $C$. By \Cref{lem: pi properties} and \Cref{lem:BoolBool} respectively, both $\im(\pi_{\psi(C)})$ and the span of $C$ are Boolean lattices. Consequently, they must coincide. Thus, the antichain $C$ consists precisely of the coatoms of $\im(\pi_{\psi(C)})$. This shows that $\varphi\circ\psi (C) = C$.

Conversely, let $T = ([n], E)$ be in $\mathrm{Tr}_n$. The atoms of the sublattice $\im(\pi_T)$ are atoms of the ambient lattice $\cP_n$. As previously mentioned, they are partitions consisting of $n-2$ blocks of size $1$ and one block of size $2$. The connected subgraphs of $T$ with two vertices are precisely its edges. This shows that $\psi\circ \varphi(T)=T$.

Thus, we know that $\varphi$ and $\psi$ are inverses of each other. We conclude by pointing out that the number of labelled trees with $n$ vertices is $n^{n-2}$, as given by Cayley's formula \cite{Cayley1897trees}.
\end{proof}

\begin{remark}\label{rem: genPartition}
    It was pointed out to us that the result on partition lattices stated in \Cref{thm: main_thm_partition_lattice} can be generalized in the following way. Consider a graph $G$ with $n$ vertices. 
    By \cite[Section 1.7]{Oxley}, the set of \textit{flats of the graphic matroid} $\mathcal{L}(M(G))$ consists of all subgraphs $S$ of $G$ with the same vertex set, in which each connected component of $S$ is an induced subgraph of $G$. Ordered by edge-inclusion, $\mathcal{L}(M(G))$ forms a lattice \cite[Lemma 1.7.3]{Oxley}, which is called the \textit{lattice of flats} of $M(G)$. Setting $G$ to be the complete graph $K_n$, one retrieves the partition lattice $\cP_n$ as the lattice of flats $\mathcal{L}(M(G))$. We expect that the bijection given in \Cref{thm: main_thm_partition_lattice} can be generalized to a bijection between Boolean antichains of size $n-1$ in the lattice of flats $\mathcal{L}(M(G))$ for an arbitrary graph $G$ and labelled spanning trees of $G$.
\end{remark}

\begin{example}\label{exm: partition lattice}
    Consider the partition lattice $\cP_4$ and the Boolean antichain consisting of the partitions $\{\{1,2,3\},\{4\}\}$,$\{\{2\},\{1,3,4\}\}$ and $\{\{1\},\{2,3,4\}\}$. This Boolean antichain generates the Boolean sublattice of $\cP_4$ shown in \Cref{fig:Partition}. One can verify that this lattice is equal to the lattice given by $\im(\pi_T)$, and isomorphic to the lattice given by $\cF(T)$ for
    $$
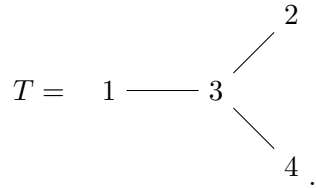
\begin{tikzpicture}
        \node (1) at (0,2) {$T = \quad 1$};
        \node (2) at (3,3) {$2$};
        \node (3) at (2,2) {$3$};
        \node (4) at (3,1) {$4$};

        \draw (1) -- (3);
        \draw (3) -- (2);
        \draw (3) -- (4);
    \end{tikzpicture}.$$
    This illustrates \Cref{lem: pi properties}. The edge set of $T$ can be read off from the blocks of size $2$ of the elements covering $\hat{0}$ in \Cref{fig:Partition}.
\end{example}

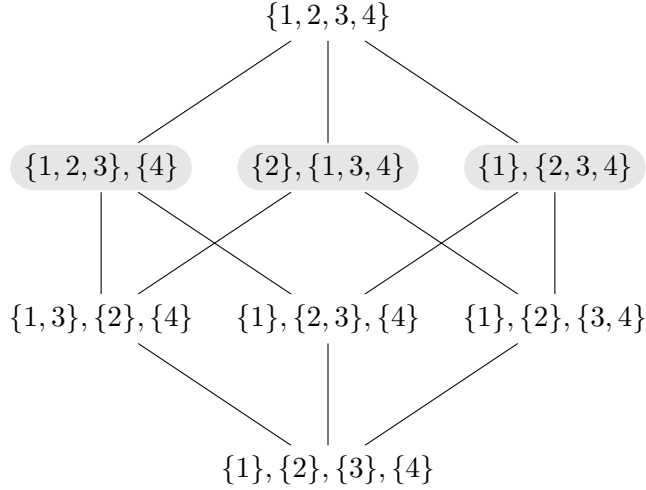
\begin{figure}
    \centering
    \begin{tikzpicture}
        \fill[rounded corners = 3mm, fill=llgray] (-1.2, 3.7) rectangle (1.2, 4.3);
        \fill[rounded corners = 3mm, fill=llgray] (1.8, 3.7) rectangle (4.2, 4.3);
        \fill[rounded corners = 3mm, fill=llgray] (4.8, 3.7) rectangle (7.2, 4.3);
        \node (0) at (3, 0) {$\{1\},\{2\},\{3\},\{4\}$};
        \node (1) at (0, 2) {$\{1,3\},\{2\},\{4\}$};
        \node (2) at (3, 2) {$\{1\},\{2,3\},\{4\}$};
        \node (3) at (6, 2) {$\{1\},\{2\},\{3,4\}$};
        \node (4) at (0, 4) {$\{1,2,3\},\{4\}$};
        \node (5) at (3, 4) {$\{2\},\{1,3,4\}$};
        \node (6) at (6, 4) {$\{1\},\{2,3,4\}$};
        \node (7) at (3, 6) {$\{1,2,3,4\}$};
        
        \draw (0) -- (1);
        \draw (0) -- (2);
        \draw (0) -- (3);
        \draw (1) -- (4);
        \draw (1) -- (5);
        \draw (2) -- (4);
        \draw (2) -- (6);
        \draw (3) -- (5);
        \draw (3) -- (6);
        \draw (4) -- (7);
        \draw (5) -- (7);
        \draw (6) -- (7);
    \end{tikzpicture}
    \caption{\centering The Boolean poset generated by a Boolean antichain in $\cP_4$}
    \label{fig:Partition}
\end{figure}

To conclude this section, we regard a specific subfamily of partitions.
Let $\varrho$ be a partition in $\cP_n$ such that there are no blocks $\varrho_1$ and $\varrho_2$ and $i<j<\ell<m \in [n]$ with $i,\ell \in \varrho_1$ and $j,m \in \varrho_2$. Then $\varrho$ is called a \textit{non-crossing partition}. We denote the set of all non-crossing partitions of the set $[n]$ by $\cP_n^{\text{NC}}$. Ordering $\cP_n^{\text{NC}}$ by refinement forms a lattice. It is a subposet of $\cP_n$, but not a sublattice. Similarly, we recall that a labelled tree $T$ is called a \textit{non-crossing labelled tree} if there are no edges $e_1=\{i,\ell\}$ and $e_2=\{j,m\}$ with $i<j<\ell<m \in [n]$.

We now show that the bijective maps $\varphi$ and $\psi$, defined in \Cref{lem: phi} and \Cref{lem: psi} respectively, restrict to bijective maps between Boolean antichains in the lattice of non-crossing partitions $\cP_n^{\text{NC}}$ and non-crossing labelled trees, thereby enumerating $\cC_{n-1}(\cP_n^{\text{NC}})$.
\begin{corollary}
    The number of Boolean antichains of size $n-1$ in the lattice of non-crossing partitions $\cP_n^{\text{NC}}$ is $\frac{1}{2n-1} \binom{3n-3}{n-1}.$
\end{corollary}

\begin{proof}
First, let $C$ be a Boolean antichain of size $n-1$ in $\cP_n^{NC}$. Similarly to \Cref{lem: psi}, we consider the atoms of the span of $C$, which again consist of $n-2$ blocks of size $1$ and one block $e_a$ of size $2$. If the edges which correspond to the blocks $e_a$ and $e_b$ were to cross, the join $\varrho'$ of the corresponding partitions would be a non-crossing partition consisting of $n-4$ blocks of size $1$ and one block of size $4$. The height $h(\varrho)$ of a partition $\varrho=\{\varrho_1,\dots,\varrho_k\}$ in $\cP_n^{NC}$ is given by $h(\varrho)=n-k$ \cite{Kreweras}. Since the span of $C$ is a Boolean sublattice of height $n-1$, $\varrho'$ must be of height $2$. This yields a contradiction, showing that $\psi(C)$ is a non-crossing labelled tree.

On the other hand, let $F$ be a subforest of $T$ and let $\varrho = \pi_T(F)$. Suppose that $\varrho$ is a crossing partition, that is, there exist $i< j< \ell< m$ in $[n]$ and blocks $\varrho_1$ and $\varrho_2$ of $\varrho$ such that $i, \ell\in \varrho_1$ and $j, m\in \varrho_2$. We show that $T$ is a crossing labelled tree. 
Since $i$ and $\ell$ lie in the same connected component of $F$, there exists a path $p_1$ in $F$ from $i$ to $\ell$. Analogously, there exists a path $p_2$ from $j$ to $m$. We observe that $j$ lies in the interval $[i, \ell]$, while $m$ lies in $[n]\setminus[i,\ell]$. Thus, the path $p_2$ must contain an edge 
\[\{j',m'\} \text{ with } j'\in [i, \ell] \text{ and } m'\in[n]\setminus[i, \ell].\]
We assume that $m'$ lies in the interval $[\ell,n]$. The case where $m'$ lies in the interval $[1,i]$ can be treated in a similar way and is left to the reader. A visual sketch is shown in \Cref{fig:Crossing Tree}, where dashed lines represent paths and bold lines represent edges. Under this assumption, $\ell$ lies in the interval $[j', m']$, while $i$ lies in $[n]\setminus[j', m']$. Hence, there exists an edge 
\[\{i',\ell'\} \text{ with } \ell'\in [j', m'] \text{ and } i'\in[n]\setminus[j', m']\]
in $p_1$. Again, we may assume that $i'$ lies in the interval $[1,j']$, leaving the case where $i'$ lies in $[m',n]$ to the reader. Since $\rho_1$ and $\rho_2$ are disjoint, the strict inequalities $i'< j'< \ell'< m'$ hold, making $T$ a crossing labelled tree.

\usetikzlibrary{decorations.pathmorphing}
\begin{figure}
\centering
\begin{minipage}{0.3\textwidth}
\begin{tikzpicture}[ampersand replacement=\&]
  \node (1) at (90:2cm) {$1$};
  \node (i') at (54:2cm) {$i'$};
  \node (i) at (18:2cm) {$i$};
  \node (j) at (-18:2cm) {$j$};
  \node (j') at (-54:2cm) {$j'$};
  \node (l) at (-90:2cm) {$\ell$};
  \node (l') at (-126:2cm) {$\ell'$};
  \node (m) at (-162:2cm) {$m$};
  \node (m') at (-198:2cm) {$m'$};
  \node (n) at (110:2cm) {$n$};

  \draw [domain=60:84] plot ({2 * cos(\x)}, {2 * sin(\x)}); 
  \draw [domain=24:48] plot ({2 * cos(\x)}, {2 * sin(\x)}); 
  \draw [domain=-11:12] plot ({2 * cos(\x)}, {2 * sin(\x)}); 
  \draw [domain=312:336] plot ({2 * cos(\x)}, {2 * sin(\x)}); 
  \draw [domain=276:300] plot ({2 * cos(\x)}, {2 * sin(\x)}); 
  \draw [domain=240:264] plot ({2 * cos(\x)}, {2 * sin(\x)}); 
  \draw [domain=204:228] plot ({2 * cos(\x)}, {2 * sin(\x)}); 
  \draw [domain=168:192] plot ({2 * cos(\x)}, {2 * sin(\x)}); 
  \draw [domain=116:156] plot ({2 * cos(\x)}, {2 * sin(\x)}); 

  \draw[decorate, dashed] (i) to[bend left] (i');
  \draw[decorate, dashed] (l') to[bend left] (l);
  \draw[decorate, dashed] (j) to[bend right] (j');
  \draw[decorate, dashed] (m') to[bend left] (m);
  
  \draw[line width=1.5pt] (j') -- (m');
  \draw[line width=1.5pt] (i') -- (l');
\end{tikzpicture}
    \end{minipage}
    \begin{minipage}{0.3\textwidth}
    \begin{tikzpicture}[ampersand replacement=\&]
  \node (1) at (90:2cm) {$1$};
  \node (i') at (54:2cm) {$i'$};
  \node (i) at (18:2cm) {$i$};
  \node (j) at (-18:2cm) {$j$};
  \node (j') at (-54:2cm) {$j'$};
  \node (l) at (-90:2cm) {$\ell$};
  \node (l') at (-126:2cm) {$\ell'$};
  \node (m) at (-162:2cm) {$m$};
  \node (m') at (72:2cm) {$m'$};
  \node (n) at (110:2cm) {$n$};
  
  \draw [domain=78:84] plot ({2 * cos(\x)}, {2 * sin(\x)}); 
  \draw [domain=24:48] plot ({2 * cos(\x)}, {2 * sin(\x)}); 
  \draw [domain=-11:12] plot ({2 * cos(\x)}, {2 * sin(\x)}); 
  \draw [domain=312:336] plot ({2 * cos(\x)}, {2 * sin(\x)}); 
  \draw [domain=276:300] plot ({2 * cos(\x)}, {2 * sin(\x)}); 
  \draw [domain=240:264] plot ({2 * cos(\x)}, {2 * sin(\x)}); 
  \draw [domain=204:228] plot ({2 * cos(\x)}, {2 * sin(\x)}); 
  \draw [domain=116:192] plot ({2 * cos(\x)}, {2 * sin(\x)}); 
  \draw [domain=60:66] plot ({2 * cos(\x)}, {2 * sin(\x)}); 

  \draw[decorate, dashed] (i) to[bend left] (i');
  \draw[decorate, dashed] (l') to[bend left] (l);
  \draw[decorate, dashed] (j) to[bend right] (j');
  \draw[decorate, dashed] (m') to[out=215, in=0] (m);
  
  \draw[line width=1.5pt] (j') -- (m');
  \draw[line width=1.5pt] (i') -- (l');
\end{tikzpicture}
    \end{minipage}
\caption{\centering Sketches of the construction of crossing edges}
\label{fig:Crossing Tree}
\end{figure}
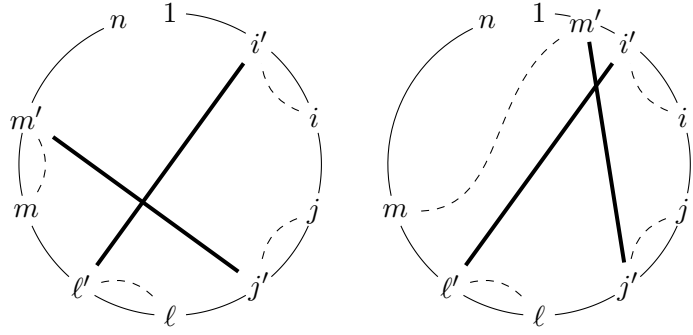
\usetikzlibrary{patterns}

Let $T$ now be a non-crossing labelled tree. The discussion above shows that $\im(\pi_T)$ is not only a Boolean sublattice of height $n-1$ of $\cP_n$, but also of $\cP_n^{NC}$. Hence, we can apply the same argument as in the proof of \Cref{lem: phi} to see that $\varphi(T) \in  \cC_{n-1}(\cP_n^{NC})$.
Since $\psi$ and $\varphi$ are inverses of each other, the result follows from \cite{Noy1998Enumeration}, where the cardinality of non-crossing labelled trees is given.
\end{proof}

\section{Tamari lattices}\label{sec: Tamari lattices}
In this section, the Tamari lattice of rank $n\in \mathbb{N}^+$, denoted by $\cT_n$, will be described in terms of parenthesized expressions on $n+1$ letters $x_1, \dots, x_{n+1}$, following \cite{tamari}. We also refer to parenthesized expressions in $\cT_n$ as \emph{words}.
A parenthesized \emph{subexpression} is a non-empty expression consisting of either a single letter or a pair of parentheses enclosing two parenthesized subexpressions. A word can uniquely be written as \[w = (e_1(\dots (e_{l-1}(e_lx_{n+2}))\dots)\]where $e_1\dots e_l$ are subexpressions of $w$. 
Cover relations are given by reparenthesizations of subexpressions following the rule
\begin{equation}\label{eq:coverRel}
(\dots((xy)z)\dots) \lessdot (\dots (x(yz))\dots).
\end{equation}
The cardinality of $\cT_n$ is given by the $n^{th}$ Catalan number
\[
    \mathtt{c}_n = \frac{1}{n+1}\binom{2n}{n}.
\]
We give a recursive construction of Boolean antichains in $\cT_{n+1}$ using a lattice congruence. Moreover, we describe $\cC_{n-1}(\cT_n)$ by establishing a bijection with a combinatorial object introduced by us.

\subsection{A lattice congruence}\label{sec: lattice congruence}
We define a map $q_n : \cT_{n+1} \to \cT_{n}$ by setting $q_n(w)$ to be the parenthesized expression obtained by removing the letter $x_{n+2}$ from $w$ and deleting the pair of redundant parentheses. The boxed entries on the right-hand side of \Cref{fig:HasseT3_4} illustrate the preimage of the element $((x_1x_2)(x_3x_4))$ under $q_3$. We omit the subscript of $q_n$ if it is clear from the context.

\begin{figure}
    \centering
    \begin{minipage}{0.3\textwidth}
    \resizebox{1\textwidth}{!}{
    \begin{tikzpicture}
        \node (0) at (2, 10.5) {$(x_1(x_2(x_3x_4)))$};
        \node (1) at (0, 7.5) {$(x_1((x_2x_3)x_4))$};
        \node (2) at (4, 6) {$((x_1x_2)(x_3x_4))$};
        \node (3) at (0, 4.5) {$((x_1(x_2x_3))x_4)$};
        \node (4) at (2, 1.5) {$(((x_1x_2)x_3)x_4)$};
        
        \draw (0) -- (1);
        \draw (0) -- (2);
        \draw (1) -- (3);
        \draw (2) -- (4);
        \draw (3) -- (4);
        \draw[rounded corners = 3mm] (2.5, 5.7) rectangle (5.5, 6.3);
    \end{tikzpicture}}
    \end{minipage}
    \hfill
    \begin{minipage}{0.6\textwidth}
    \resizebox{1\textwidth}{!}{
    \begin{tikzpicture}
        \draw[rounded corners = 3mm](6, 4) rectangle (9, 11);
        \node (0) at (4.5, 13.5) {\footnotesize$(x_1(x_2(x_3(x_4 x_5))))$};
        \node (1) at (4.5, 10.5) {\footnotesize$(x_1(x_2((x_3x_4)x_5)))$};
        \node (2) at (7.5, 10.5) {\footnotesize$((x_1x_2)(x_3(x_4x_5)))$};
        \node (3) at (0, 9) {\footnotesize$(x_1((x_2x_3)(x_4x_5)))$};
        \node (4) at (3, 9) {\footnotesize$(x_1((x_2(x_3x_4))x_5))$};
        \node (5) at (1.5, 7.5) {\footnotesize$(x_1(((x_2x_3)x_4)x_5))$};
        \node (6) at (7.5, 7.5) {\footnotesize$((x_1x_2)((x_3x_4)x_5))$};
        \node (7) at (0, 6) {\footnotesize$((x_1(x_2x_3))(x_4x_5))$};
        \node[fill = white] (8) at (6, 6) {\footnotesize$((x_1(x_2(x_3x_4)))x_5)$};
        \node (9) at (4.5, 4.5) {\footnotesize$((x_1((x_2x_3)x_4))x_5)$};
        \node (10) at (7.5, 4.5) {\footnotesize$(((x_1x_2)(x_3x_4))x_5)$};
        \node (11) at (0, 3) {\footnotesize$(((x_1x_2)x_3)(x_4x_5))$};
        \node (12) at (3, 3) {\footnotesize$(((x_1(x_2x_3))x_4)x_5)$};
        \node (13) at (3, 0) {\footnotesize$((((x_1x_2)x_3)x_4)x_5)$};
        
        \draw (0) -- (1);
        \draw (0) -- (2);
        \draw (0) -- (3);
        \draw (1) -- (4);
        \draw (1) -- (6);
        \draw (2) -- (11);
        \draw (2) -- (6);
        \draw (3) -- (5);
        \draw (3) -- (7);
        \draw (4) -- (5);
        \draw (4) -- (8);
        \draw (5) -- (9);
        \draw (6) -- (10);
        \draw (7) -- (11);
        \draw (7) -- (12);
        \draw (8) -- (9);
        \draw (8) -- (10);
        \draw (9) -- (12);
        \draw (10) -- (13);
        \draw (11) -- (13);
        \draw (12) -- (13);
        
    \end{tikzpicture}}
    \end{minipage}
    \caption{\centering Hasse diagrams of $\cT_3$ and $\cT_4$, illustrating the preimage of the element $((x_1, x_2)(x_3, x_4))$ under $q_3$}
    \label{fig:HasseT3_4}
\end{figure}
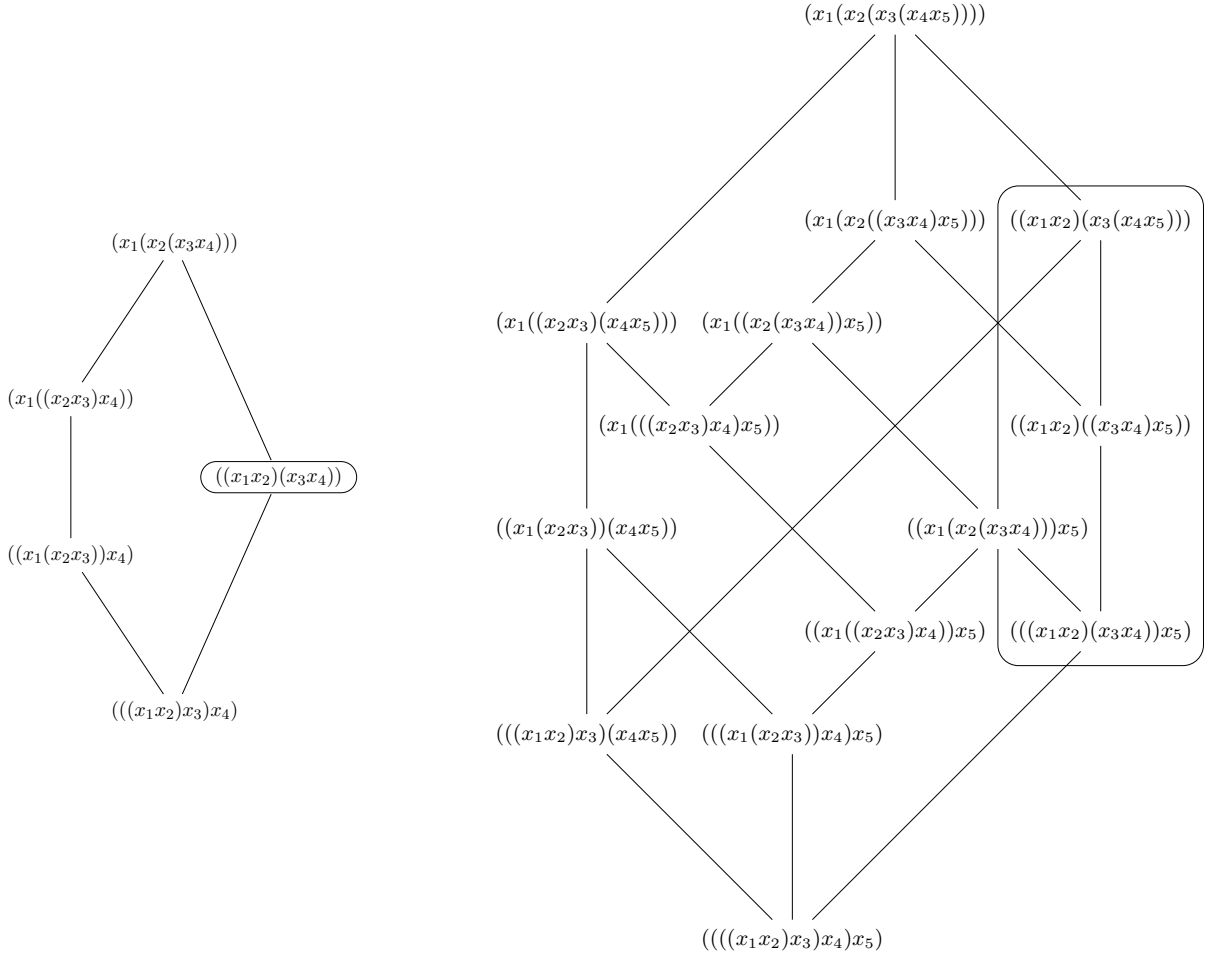

Following \cite[Section 2]{Reading}, we say that an equivalence relation $\cR$ on a lattice $L$ is a \emph{lattice congruence} if and only if it satisfies the following three conditions.
\begin{enumerate}[itemsep = 0pt]
    \item \label{item:latCong1} For all $a\in L$, the equivalence classes $[a]_\cR$ are intervals.
    \item \label{item:latCong2} The map $\topmap \colon L \to L$ defined by $\topmap(a) := \max [a]_\cR$ is order-preserving.
    \item \label{item:latCong3} The map $\botmap \colon L \to L$ defined by $\botmap(a) := \min [a]_\cR$ is order-preserving.
\end{enumerate}

Defining $[w] = q_n^{-1}(\{q_n(w)\})$ to be the equivalence class of a word $w$ yields an equivalence relation on $\cT_{n+1}$. The aim of this subsection is to prove that this equivalence relation is a lattice congruence $\equiv$ (see \Cref{prop: LatticeCongruence}), which will provide us with a useful isomorphism $ \cT_{n+1} /_\equiv \longleftrightarrow \cT_n$.
We begin by explicitly describing these equivalence classes.

\begin{lemma}\label{lem:chain}
    Let $w$ be an element of $\cT_{n+1}$. Then $[w]$ is a chain (\emph{i.e.} totally ordered).
\end{lemma}
\begin{proof}
Let $k$ be the number of consecutive closing parentheses at the end of $q(w)$. Then there are parenthesized subexpressions $e_1, \dots, e_k$ such that
\[
    q(w) = (e_1(e_2(\dots(e_{k-1}(e_kx_{n+1}))\dots).
\]
The letter $x_{n+2}$ can be added before or after each of these closing parentheses, which gives $k+1$ words in $\cT_{n+1}$. These are ordered by the number of closing parentheses:
\[
    (e_1(\dots(e_{i-1}((e_i(\dots(e_kx_{n+1}\underbrace{)\dots)}_{k-i+1}x_{n+2}\underbrace{)\dots)}_{i} \lessdot (e_1(\dots(e_i((e_{i+1}(\dots(e_kx_{n+1}\underbrace{)\dots)}_{k-i}x_{n+2}\underbrace{)\dots)}_{i+1}
\]
\end{proof}
We define the \emph{height} of $w$ to be the number of words that lie strictly below $w$ in $[w]$, and denote it by $h(w)$. We use the notation $[w]_i$ to denote the element of height $i$ in $[w]$. By the proof above, we see that $h(w)$ is the number of closing parentheses at the end of $w$ minus one. This observation leads to the following result.

\begin{corollary}\label{size of equivalence class}
Let $w$ be an element of $\cT_{n+1}$, then \[|[w]| = h(q(w)) + 2.\]
In particular, the equivalence class $[\hat{1}]$ in $\cT_{n+1}$ has size $n+1$.
\end{corollary}

The following technical lemma provides a crucial description of the interaction between the equivalence relation defined by $q_n$ and the covering relations of the lattices $\cT_n$ and $\cT_{n+1}$.

\begin{lemma}[Banana Lemma\footnote{This name stems from the habit of the first four authors of referring to the equivalence classes under \emph{q} as "bananas". We found the name catchy and decided to keep it.}]\label{lem:banana}
    Let $a$ and $b$ be elements in $\cT_{n}$ such that $a \lessdot b$. Let $z$ be the subexpression on the right of the reparenthesization that transforms $a$ into $b$ as in \Cref{eq:coverRel}.
    \begin{enumerate}
        \item If $x_{n+1} \notin z$, then the element of height $i$ in $q^{-1}(\{a\})$ is covered by the element of height $i$ in $q^{-1}(\{b\})$ for every possible height $i$. In particular, both classes have the same cardinality.
        \item If $x_{n+1} \in z$, let $j$ be the number of closing parentheses after $z$. Then the element of height $i$ in $q^{-1}(\{a\})$ is covered by the element of height $i$ in $q^{-1}(\{b\})$ if $i < j$ and by the element of height $i+1$ otherwise. In particular, $|q^{-1}(\{b\})| = |q^{-1}(\{a\})| + 1$.
    \end{enumerate}
\end{lemma}

\begin{figure}
    \centering
    \scalebox{0.9}{
    \begin{tikzpicture}
    \draw[rounded corners = 3mm, gray] (6.5, 10) rectangle (9.5, 14);
    \draw[rounded corners = 3mm, gray] (11.5, 8) rectangle (14.5, 11);
    \draw[rounded corners = 3mm, gray] (2.5, 8.5) rectangle (5.5, 11.5);
    \draw[rounded corners = 3mm, gray] (-1.5, 7) rectangle (1.5, 9);
    \draw[rounded corners = 3mm, gray] (4.5, 4.5) rectangle (7.5, 6.5);
    \node (0) at (8, 13.5) {\footnotesize$(x_1(x_2(x_3(x_4 x_5))))$};
    \node (1) at (8, 12.5) {\footnotesize$(x_1(x_2((x_3x_4)x_5)))$};
    \node (4) at (8, 11.5) {\footnotesize$(x_1((x_2(x_3x_4))x_5))$};
    \node (8) at (8, 10.5) {\footnotesize$((x_1(x_2(x_3x_4)))x_5)$};
    
    \node (2) at (13, 10.5) {\footnotesize$((x_1x_2)(x_3(x_4x_5)))$};
    \node (6) at (13, 9.5) {\footnotesize$((x_1x_2)((x_3x_4)x_5))$};
    \node (10) at (13, 8.5) {\footnotesize$(((x_1x_2)(x_3x_4))x_5)$};
    
    \node (3) at (4, 11) {\footnotesize$(x_1((x_2x_3)(x_4x_5)))$};
    \node (5) at (4, 10) {\footnotesize$(x_1(((x_2x_3)x_4)x_5))$};
    \node (9) at (4, 9) {\footnotesize$((x_1((x_2x_3)x_4))x_5)$};
    
    \node (7) at (0, 8.5) {\footnotesize$((x_1(x_2x_3))(x_4x_5))$};
    \node (12) at (0, 7.5) {\footnotesize$(((x_1(x_2x_3))x_4)x_5)$};

    \node (11) at (6, 6) {\footnotesize$(((x_1x_2)x_3)(x_4x_5))$};
    \node (13) at (6, 5) {\footnotesize$((((x_1x_2)x_3)x_4)x_5)$};

    \draw (0) -- (1);
    \draw (0) -- (2);
    \draw (0) -- (3);
    \draw (1) -- (4);
    \draw (1) -- (6);
    \draw (2) -- (11);
    \draw (2) -- (6);
    \draw (3) -- (5);
    \draw (3) -- (7);
    \draw (4) -- (5);
    \draw (4) -- (8);
    \draw (5) -- (9);
    \draw (6) -- (10);
    \draw (7) -- (11);
    \draw (7) -- (12);
    \draw (8) -- (9);
    \draw (8) -- (10);
    \draw (9) -- (12);
    \draw (10) -- (13);
    \draw (11) -- (13);
    \draw (12) -- (13);
    \end{tikzpicture}}
    \caption{Illustration of \Cref{lem:banana}}
    \label{fig:placeholder}
\end{figure}

\begin{proof}
    Let $k$ be the number of closing parentheses at the end of $a$.
    \begin{enumerate}
        \item We can write
        \[
            a = (e_1(\dots(e_{j-1}(e_j(e_{j+1}(\dots(e_kx_{n+1})\dots)
        \]
        for some parenthesized expressions $e_1,\dots e_k$. Since $x_{n+1}$ is not contained in $z$, the reparenthesization occurs inside of one of these expressions and we may assume this expression to be $e_j$. We call the resulting expression $\tilde{e_j}$. Then we can write
        \begin{align*}
        b = (e_1(\dots(e_{j-1}(\tilde{e_j}(e_{j+1}(\dots(e_kx_{n+1})\dots).
        \end{align*}
        Adding $x_{n+2}$, we see that elements of the same height only differ by $e_j$ and $\tilde{e_j}$. Thus, we have a covering relation between elements of the same height in $q^{-1}(a)$ and $q^{-1}(b)$.
        \item 
        Again, we write $a$ as above and see that $z$ is of the form $(e_{j+1}(\dots(e_kx_{n+1})\dots)$ for some $j$. Since $z$ is a subexpression as in \eqref{eq:coverRel}, we know that $e_j = (xy)$. We get
        \begin{align*}
            a &= (e_1(\dots(e_{j-1}((xy)\overbrace{(e_{j+1}(\dots(e_k x_{n+1}\underbrace{)\dots)}_{k-j}}^{z}\underbrace{)\dots)}_{j} \text{ and} \\
            b &= (e_1(\dots(e_{j-1}(x(y(e_{j+1}(\dots(e_k x_{n+1}\underbrace{)\dots)}_{k-j} \, ) \underbrace{)\dots)}_{j}.
        \end{align*}
        Suppose that $i < j$. We have the following covering relation between the element of height $i$ in $q^{-1}(\{a\})$ and the element of height $i$ in $q^{-1}(\{b\})$. The shifting parentheses are marked.
        \begin{align*}
            &(e_1(\dots(e_i((e_{i+1}(\dots(e_{j-1}(\pmb(xy\pmb)(e_{j+1}(\dots(e_k x_{n+1}\underbrace{)\dots)}_{k-j}\underbrace{)\dots)}_{j-i}x_{n+2}\underbrace{)\dots)}_{i+1} \\
            \lessdot &(e_1(\dots(e_i((e_{i+1}(\dots(e_{j-1}(x\pmb(y(e_{j+1}(\dots(e_k x_{n+1}\underbrace{)\dots)}_{k-j} \pmb) \underbrace{)\dots)}_{j-i}x_{n+2}\underbrace{)\dots)}_{i+1}.
        \end{align*}
        For height $i \geq j$ in $q^{-1}(\{a\})$, we obtain a shift of heights:
        \begin{align*}
            &(e_1(\dots(e_{j-1}(\pmb(xy\pmb)(e_{j+1}(\dots(e_i((e_{i+1}(\dots((e_k x_{n+1}\underbrace{)\dots)}_{k-i}x_{n+2}\underbrace{)\dots)}_{i-j+1}\underbrace{)\dots)}_{j} \\
            \lessdot &(e_1(\dots(e_{j-1}(x\pmb(y(e_{j+1}(\dots(e_i((e_{i+1}(\dots(e_k x_{n+1}\underbrace{)\dots)}_{k-i}x_{n+2}\underbrace{)\dots)}_{i-j+1}\,\pmb)\underbrace{)\dots)}_{j}.
        \end{align*}
    \end{enumerate}
\end{proof}

With this, we are now able to prove the main result of this subsection.

\begin{proposition}\label{prop: LatticeCongruence}
    The map $q$ defines a lattice congruence on $\cT_{n+1}$.
\end{proposition}
\begin{proof}
  By construction, we have that $a\leq b$ implies $q(a)\leq q(b)$. This shows that $[a]$ is an interval since $\botmap(a)\leq b \leq \topmap(a)$ implies $q(a)=q(b)$. Moreover, let $a \lessdot b \in \cT_{n+1}$ with $[a]\neq [b]$. Then \Cref{lem:banana} yields $\botmap(a)\lessdot \botmap(b)$ and $\topmap(a)\lessdot\topmap(b)$. From this, we deduce that $\botmap$ and $\topmap$ are order-preserving.
\end{proof}

\begin{remark}\label{rem: banana lemma}
By \cite[Section 2]{Reading}, the quotient $\cT_{n+1} /_\equiv \simeq \cT_n$ is isomorphic to the induced subposets $\topmap(\cT_{n+1})$ and $\botmap(\cT_{n+1})$ of $\cT_{n+1}$. By a repeated application of \Cref{lem:banana} we also obtain that bottom elements of equivalence classes only cover other bottom elements and that top elements of an equivalence class are covered by other top elements. Elements that are neither top nor bottom elements are covered by neither top nor bottom elements of different equivalence classes. 
\end{remark}

Combining this with \Cref{prop: LatticeCongruence} gives the following result.

\begin{corollary}\label{cor: bot and top meet and join}
    Let $a,b \in \cT_{n+1}$. Then
    \begin{enumerate}
        \item $\topmap(a)\wedge \topmap(b) = \topmap(a\wedge b)$, $\topmap(a)\vee \topmap(b) = \topmap(a\vee b)$,
        \item $\botmap(a)\wedge \botmap(b) = \botmap(a\wedge b)$ and $\botmap(a)\vee \botmap(b) = \botmap(a\vee b)$.
    \end{enumerate}
\end{corollary}

\subsection{Closing index sets}\label{sec: closing index sets}
In this subsection, we give structural results on the heights of elements in their equivalence class. This will help us to show that certain antichains are Boolean.
We use the idea behind a well-known bijection between the set of closing parentheses of a parenthesized expression, down steps in a Dyck path and a $312$-avoiding permutation (cf. \cite[Section 2.2.1]{Knuth}). This leads to the following definition.

\begin{definition}\label{def: closing parenthesis}
    Let $v \in \cT_{n+1}$ and $k \in [n+1]$. Let $z_k^v$ be the smallest parenthesized subexpression of $v$ such that $x_k \in z_k^v$ and such that there exists another subexpression $y^v$ with $(z_k^vy^v)$ a subexpression of $v$. Then we call the closing parenthesis of $(z_k^vy^v)$ the \textit{closing parenthesis corresponding to $x_k$} and denote it by $cp^v_k$. With this, let
    $$I_v:=\{k \in [n]: cp_k^v \text{ closes before } x_{n+2}\}$$
    be the \textit{closing index set} of $v$.
\end{definition}

As the definition suggests, there exists no closing parenthesis $cp^v_{n+2}$, and the closing parenthesis $cp^v_{n+1}$ always closes after $x_{n+2}$ by construction.
Furthermore, \Cref{lem:chain} shows that the height of an element $v$ in its equivalence class is given by $h(v)=n - |I_v|$.

\begin{example}
    Let $u = (((x_1x_2)x_3)x_4), v=((x_1(x_2x_3))x_4)$ and $w=(x_1(x_2(x_3x_4)))$ be elements of $\cT_3$ (see \Cref{fig:HasseT3_4}). Then $cp_2^u, cp_2^v$ and $cp_2^w$ are the following closing parentheses:
    $$
        u = (((\underbrace{x_1x_2}_{z_2^u})\underbrace{x_3}_{y^u}\pmb{)}x_4), v = ((x_1(\underbrace{x_2}_{z_2^v}\underbrace{x_3}_{y^v}\pmb{)})x_4), w = (x_1(\underbrace{x_2}_{z_2^w}(\underbrace{x_3x_4}_{y^w})\pmb{)})
    $$
    The closing index sets are
    $$
        I_u = \{1,2\}, I_v = \{1,2\} \text{ and } I_w = \{\}.
    $$
\end{example}

By \Cref{lem:chain}, we know that the closing index sets of bottom elements of $q$-equivalence classes in $\cT_{n+1}$ always contain all indices from $1$ to $n$. 

For $v \in \cT_{n+1}$, we write
\[ k_v :=
    \begin{cases*}
    0 & if $I_v = \{1,\dots,n\}$ and\\
      \max([n] \setminus I_v) & otherwise.
    \end{cases*}
\]

This allows us to describe how closing index sets change along a cover relation.

\begin{lemma}\label{lem: closing parentheses in equivalence class}
    Let $v\lessdot w \in \cT_{n+1}$ and let $z$ be the subexpression on the right of the reparenthesization that transforms $v$ into $w$ as in \Cref{eq:coverRel}. If $x_{n+2}\notin z$, then $I_v = I_w$. If $x_{n+2}\in z$, then $I_w\subsetneq I_v$. Additionally, if $v$ and $w$ are elements of the same equivalence class, we have $I_v = I_w \cup \{k_w\}$.
\end{lemma}
\begin{proof}
    If $x_{n+2}\notin z$, we have $I_v=I_w$ by construction. If $x_{n+2} \in z$, we know, similarly to the proof of \Cref{lem:banana}, that there exist parenthesized subexpressions $e_1,\dots,e_j$ such that
    \begin{align*}
        &v = (e_1(\dots(e_{i-1}((e_ie_{i+1}\pmb{)}\overbrace{(e_{i+2}(\dots(e_j x_{n+2})\dots))}^z)\dots) \text{ and } \\
        &w = (e_1(\dots(e_{i-1}(e_i(e_{i+1}(e_{i+2}(\dots(e_j x_{n+2})\dots))\pmb{)})\dots).
    \end{align*}
    The marked parenthesis in $v$ is the closing parenthesis corresponding to the last letter in $e_i$, which we call $x_\ell$. Thus, the closing parenthesis corresponding to $x_\ell$ in $w$ appears immediately after the marked parenthesis. The closing parentheses corresponding to the indices $[n]\setminus\{\ell\}$ do not change their position relative to the letters under the reparenthesization from $v$ to $w$.
    This shows that $I_v = I_w \cup \{\ell\}$. 

    If $v$ and $w$ are elements of the same equivalence class, then by \Cref{lem:chain}, we have $z=x_{n+2}$.
    Since, in this case, the second closing parenthesis after $x_{n+2}$ in $w$ corresponds to $x_{k_w}$ by definition, it follows that $\ell = k_w$. This shows the remaining assertion.
\end{proof}

As an immediate consequence, we can expand this result to all relations in $\cT_{n+1}$.

\begin{corollary}\label{cor: closing parenthesis moved to the right}
    Let $v \leq w \in \cT_{n+1}$, then $I_w \subseteq I_v$.
\end{corollary}

With this, we can describe the closing index set of the meet of two elements in $\cT_{n+1}$.

\begin{lemma}\label{lem: meet height is union}
    Let $v,w \in \cT_{n+1}$. Then we have $I_{v \wedge w} = I_v \cup I_w$.
\end{lemma}
\begin{proof}
    Since $v\wedge w \leq v$ and $v\wedge w \leq w$, we have by \Cref{cor: closing parenthesis moved to the right} that $I_v \cup I_w \subseteq I_{v\wedge w}$. \\
    On the other hand, suppose $\ell \in I_{v \wedge w}\setminus(I_v \cup I_w)$, and let 
    $$a:= \Bigl(\underbrace{(\dots(}_{\ell-1}x_1x_2)x_3)\ldots)x_\ell)\underbrace{(\dots(}_{n-\ell+1}x_{\ell+1}x_{\ell+2})x_{\ell+3})\ldots)x_{n+2})\underset{\substack{\uparrow \\ cp_\ell^a}}{\Bigr)}.$$
    Since $\ell \notin I_v$, there exist parenthesized subexpressions $e_1,\dots,e_j$ such that
    \[
        v = (e_1(\dots(e_{i-1}(e_i(e_{i+1}(\dots(e_j x_{n+2} \underbrace{)\dots)}_j,
    \]
    where $x_\ell$ is the last letter in $e_i$. Let $z$ and $z'$ be subexpressions of the form
    \[
        z=(e_{i+1}(\dots(e_j x_{n+2} \underbrace{)\dots)}_{j-i}, z'=\underbrace{(\dots(}_{n+1-\ell}x_{\ell+1}x_{\ell+2}) x_{\ell+3}) \dots )x_{n+1})x_{n+2}).
    \] 
    We see that $z$ is a subexpression of $v$, and that both $z$ and $z'$ are parenthesized expressions on the $n-\ell+2$ letters $x_{\ell+1}, \dots, x_{n+2}$. If we rename $x_i$ as $x_{i-\ell}$ for $\ell+1 \leq i \leq n+2$, we can view both expressions as elements of a Tamari lattice $\cT_{n-\ell+1}$. We see that $z'$ becomes the least element of $\cT_{n-\ell+1}$. Thus, $z \geq z'$ holds, which yields
    \[
        v \geq (e_1(\dots(e_{i-1}(e_iz' \underbrace{)\dots)}_i=:v'.
    \]
    The expressions $e_1, \dots, e_i$ are parenthesized expressions on the letters $x_1, \dots, x_\ell$. Treating $z'$ like an additional letter yields an element of a Tamari lattice $\cT_{\ell}$.
    Comparing that to the least element, we get
    \begin{align*}
        v' &\geq \underbrace{(\dots(}_{\ell}x_1x_2)x_3)\dots)x_\ell)z') \\
        &= \Bigl( \underbrace{(\dots(}_{\ell-1}x_1x_2)x_3)\dots)x_\ell) \underbrace{(\dots(}_{n+1-\ell}x_{\ell+1}x_{\ell+2}) x_{\ell+3}) \dots )x_{n+1})x_{n+2}) \Bigr) = a
    \end{align*}
    
    Analogously, the inequality $a \leq w$ holds, and likewise does $v\wedge w \geq a$. Thus, \Cref{cor: closing parenthesis moved to the right} shows that $\ell \in I_a$. This yields a contradiction since we have $I_a = \{1,\dots,n\}\setminus \{\ell\}$.
\end{proof}


We now give a recursive description of all Boolean antichains in $\cT_{n+1}$ using the following order-preserving map:
\begin{align*}
    f_n: \cT_n &\rightarrow \cT_{n+1} \\
    a &\mapsto \max(q_n^{-1}(\{a\})).
\end{align*}

Since $q_n$ defines a lattice congruence, $f_n$ is a lattice monomorphism. 
Moreover, let $e_1,\dots,e_k$ be subexpressions of $v \in \cT_n$ such that
$$
    v = (e_1(e_2(\dots(e_{k-1}(e_kx_{n+1}))\dots).
$$
Then \Cref{lem:chain} shows that $f_n(v) \in \cT_{n+1}$ is of the form
$$
    f_n(v) = (e_1(e_2(\dots(e_{k-1}(e_k(x_{n+1}x_{n+2})))\dots).
$$
This yields that $I_{f_n(v)} = I_v$. With this, we can provide the following characterisation of top elements of equivalence classes.

\begin{lemma}\label{cor: k_w = n}
    Let $w \in \cT_{n+1}$. Then the following three assertions are equivalent.
    \begin{enumerate}
        \item\label{(1)} $w$ is the top element of its equivalence class.
        \item\label{(2)} The closing index set $I_w$ does not contain the index $n$.
        \item\label{(3)} We have $k_w = n$.
    \end{enumerate}
\end{lemma}
\begin{proof}
    Using $I_{f_n(v)} = I_v$ for $w=f_n(v)$ shows that \eqref{(1)} implies \eqref{(2)}. Statement \eqref{(3)} follows from \eqref{(2)} by definition.
    Let $w$ now be a non-top element of its equivalence class. The element $\operatorname{top}(w)$ has $k_{\operatorname{top}(w)}=n$. By \Cref{lem: closing parentheses in equivalence class}, $w$ has $n$ in its closing index set. Thus, \eqref{(3)} implies \eqref{(1)} by contraposition.
\end{proof}

\subsection{A recursive construction}\label{sec: recursive construction}

In this section, we use the lattice congruence $q$ to construct Boolean antichains $C'$ of size $k$ in $\cT_{n+1}$ from Boolean antichains $C$ of size $k-1$ or $k$ in $\cT_n$. 
In \Cref{prop: BooleanAreValid} and \Cref{prop: ValidAreBoolean}, we will show that, under an additional technical assumption, this process yields a recursive construction of all Boolean antichains in $\cT_{n+1}$.

\begin{definition}\label{def: Almost Boolean}
    Let $C$ be an antichain of size $k\leq n$ in $\cT_{n+1}$. We call $C$ \textit{almost Boolean} if it is of the form \[C=\{f_n(c_1),\ldots,f_n(c_{k-1}),w\}\] where either $\{c_1,\ldots,c_{k-1},q_n(w)\}$ is a Boolean antichain in $\cT_n$ or $\{c_1,\ldots,c_{k-1}\}$ is a Boolean antichain in $\cT_n$ and $w \in  [\hat{1}_{\cT_{n+1}}]\setminus\{\hat{1}_{\cT_{n+1}}\}$.
    If there exists an element $w \in C$ which is not the top element of its equivalence class, we write $w(C):=w$.
\end{definition}


\begin{example}\label{ex: almost boolean}
    \Cref{fig:almostBoolean} shows an antichain of size $3$ in $\cT_4$ that is almost Boolean but not Boolean. This can easily be verified by checking that $(a\wedge w)\vee(b\wedge w) < w = \wedge(\{a,w\}\cap \{b,w\})$ for
    $$w=(x_1((x_2(x_3x_4))x_5)), a = ((x_1(x_2x_3))(x_4x_5)) \text{ and } b = ((x_1x_2)(x_3(x_4x_5))).$$
\end{example}

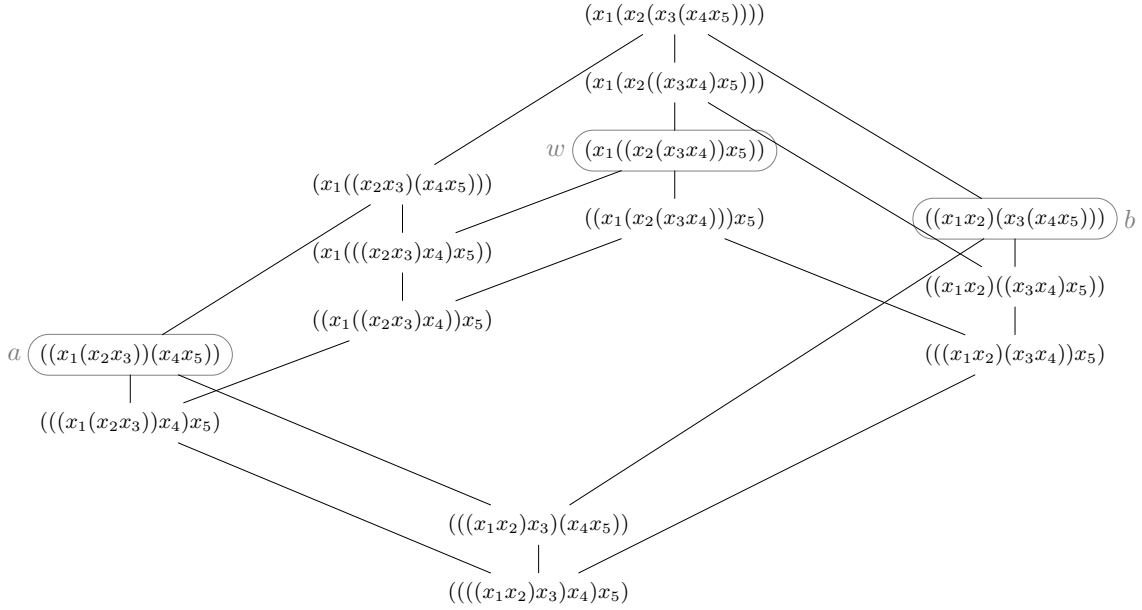
\begin{figure}
    \centering
    \scalebox{0.9}{
    \begin{tikzpicture}
    \draw[rounded corners = 3mm, gray](6.5, 11.2) rectangle (9.5, 11.8); 
    \draw[rounded corners = 3mm, gray] (11.5, 10.2) rectangle (14.5, 10.8); 
    \draw[rounded corners = 3mm, gray] (-1.5, 8.2) rectangle (1.5, 8.8); 
    
    \node (0) at (8, 13.5) {\footnotesize$(x_1(x_2(x_3(x_4 x_5))))$};
    \node (1) at (8, 12.5) {\footnotesize$(x_1(x_2((x_3x_4)x_5)))$};
    \node[label={[text=gray]180:$w$}] (4) at (8, 11.5) {\footnotesize$(x_1((x_2(x_3x_4))x_5))$};
    \node (8) at (8, 10.5) {\footnotesize$((x_1(x_2(x_3x_4)))x_5)$};
    
    \node[label={[text=gray]0:$b$}] (2) at (13, 10.5) {\footnotesize$((x_1x_2)(x_3(x_4x_5)))$};
    \node (6) at (13, 9.5) {\footnotesize$((x_1x_2)((x_3x_4)x_5))$};
    \node (10) at (13, 8.5) {\footnotesize$(((x_1x_2)(x_3x_4))x_5)$};
    
    \node (3) at (4, 11) {\footnotesize$(x_1((x_2x_3)(x_4x_5)))$};
    \node (5) at (4, 10) {\footnotesize$(x_1(((x_2x_3)x_4)x_5))$};
    \node (9) at (4, 9) {\footnotesize$((x_1((x_2x_3)x_4))x_5)$};
    
    \node[label={[text=gray]180:$a$}] (7) at (0, 8.5) {\footnotesize$((x_1(x_2x_3))(x_4x_5))$};
    \node (12) at (0, 7.5) {\footnotesize$(((x_1(x_2x_3))x_4)x_5)$};

    \node (11) at (6, 6) {\footnotesize$(((x_1x_2)x_3)(x_4x_5))$};
    \node (13) at (6, 5) {\footnotesize$((((x_1x_2)x_3)x_4)x_5)$};

    \draw (0) -- (1);
    \draw (0) -- (2);
    \draw (0) -- (3);
    \draw (1) -- (4);
    \draw (1) -- (6);
    \draw (2) -- (11);
    \draw (2) -- (6);
    \draw (3) -- (5);
    \draw (3) -- (7);
    \draw (4) -- (5);
    \draw (4) -- (8);
    \draw (5) -- (9);
    \draw (6) -- (10);
    \draw (7) -- (11);
    \draw (7) -- (12);
    \draw (8) -- (9);
    \draw (8) -- (10);
    \draw (9) -- (12);
    \draw (10) -- (13);
    \draw (11) -- (13);
    \draw (12) -- (13);
    \end{tikzpicture}}
    \caption{\centering Example of an almost Boolean antichain in $\cT_4$ that is not Boolean}
    \label{fig:almostBoolean}
    \end{figure}

By imposing an additional condition to the index $k_w$ of an almost Boolean antichain we get an equivalent condition for an antichain to be Boolean in $\cT_{n+1}$. This provides a recursive construction of all Boolean antichains in the lattice. 

\begin{proposition}\label{prop: BooleanAreValid}
    Let $C$ be a Boolean antichain of size $k\leq n$ in $\cT_{n+1}$. Then $C$ is an almost Boolean antichain of the form $C=\{f_n(c_1),\dots,f_n(c_{k-1}),w\}$, and there exists at most one $i \in [1, k-1]$ with $k_w \in I_{f_n(c_i)}$.
\end{proposition}
\begin{proof}
    We first show that $C$ is almost Boolean. For this, we suppose that there exist two distinct elements $a$ and $b$ in $C$ which are not the top of their equivalence class. There exists a chain
    $$
        q_n(a) \lessdot q_n(d_1)\lessdot \dots \lessdot q_n(d_\ell) \lessdot \hat{1}_{\cT_n}
    $$
    in $\cT_n$. Since $a\neq \topmap(a)$, applying \Cref{rem: banana lemma} on this chain implies that $a\leq [\hat{1}_{\cT_{n+1}}]_{n-1} \lessdot \hat{1}_{\cT_{n+1}}$. Similarly, it implies that $b \leq [\hat{1}_{\cT_{n+1}}]_{n-1} \lessdot \hat{1}_{\cT_{n+1}}$. This yields that $a\vee b < \hat{1}_{\cT_{n+1}}$, which contradicts \Cref{lem:single elements join to top}. Thus, we can assume that $C$ is of the form $C=\{f_n(c_1),\dots,f_n(c_{k-1}),w\}$. Using \Cref{prop: LatticeCongruence}, we see that \eqref{eq: intersectivity} is still fulfilled for $\{c_1,\dots,c_{k-1},q_n(w)\}$. Using \Cref{rem:Intersective is antichain} gives that $\{c_1,\dots,c_{k-1}\}$ is a Boolean antichain in $\cT_n$ and if $w \notin [\hat{1}_{\cT_{n+1}}]$, then $\{c_1,\dots,c_{k-1}, q_n(w)\}$ is a Boolean antichain in $\cT_n$ as well. It follows that $C$ is almost Boolean.

    Now suppose that there exist $i\neq j \in [1,k]$ such that $k_w \in I_{f_n(c_i)}\cap I_{f_n(c_j)}$. Let $S=\{w,f_n(c_i)\}$ and $S'=\{w,f_n(c_j)\}$. Furthermore, let $c_i':= [w]_{h(w)-1} \wedge f_n(c_i)$. By \Cref{lem: closing parentheses in equivalence class} we have $I_{[w]_{h(w)-1}} = I_w \cup \{k_w\}$. Hence, our assumption on $f_n(c_i)$ and \Cref{lem: meet height is union} give that $I_{\wedge S} = I_w \cup I_{f_n(c_i)} = I_{c_i'}$. Additionally, we know by \Cref{prop: LatticeCongruence} that $[\wedge S] = [c_i']$. As $\wedge S$ and $c'_i$ have the same closing index set and thus the same height, we conclude that $\wedge S = c_i'$. In particular, this gives $[w]_{h(w)-1} \geq \wedge S$. We can show similarly that $[w]_{h(w)-1} \geq \wedge S'$, which yields
    $$
    \left(\bigwedge S\right) \vee \left(\bigwedge S'\right) \leq [w]_{h(w)-1}\lessdot w = \bigwedge (S\cap S').
    $$
    This contradicts $C$ being Boolean.
\end{proof}

\begin{example}
    We consider $a, b$ and $w$ as in \Cref{ex: almost boolean} and compute the index sets. We see $I_a= \{1,2\}$, $I_b=\{1\}$ and $k_w=1$. Thus we get $k_w \in I_a\cap I_b$, showing again that $\{a,b,w\}$ is not Boolean by \Cref{prop: BooleanAreValid}.
\end{example}

\begin{remark}\label{rem: generalisation BooleanAreAlmostBoolean}
The first part of the proof of \Cref{prop: BooleanAreValid}, which shows that Boolean antichains are almost Boolean antichains, relies heavily on \Cref{lem:banana}, \Cref{prop: LatticeCongruence} and \Cref{rem: banana lemma}. By considering the properties required for this to hold in the most general framework, we arrive at the following result:
Let $L$ be a finite lattice and $\equiv$ a lattice congruence with projection map $q$, such that for all $a \lessdot b$ in $L/_\equiv$, every element of $q^{-1}(a)$ is covered by an element of $q^{-1}(b)$, and all equivalence classes are chains. Furthermore, let
\begin{align*}
    f: L/_\equiv &\rightarrow L \\
    a &\mapsto \max(q^{-1}(a)).
\end{align*}
If $C$ is a Boolean antichain of size $k$ in $L$, then $C$ is of the form $C=\{f(c_1),\ldots,f(c_{k-1}),w\}$ where $\{c_1,\ldots,c_{k-1}\}$ is a Boolean antichain in $L/_\equiv$ and $w \in  [\hat{1}]\setminus\{\hat{1}\}$ or $\{c_1,\ldots,c_{k-1},q(w)\}$ is a Boolean antichain in $L/_\equiv$. 
This gives a framework that encompasses both \Cref{prop: BooleanAreValid} and \Cref{rem: boolean lattice recursion}. For the latter, the lattice congruence is the map that removes the element $n$ from any subset of $[1, n]$.
\end{remark}

In order to prove the converse of \Cref{prop: BooleanAreValid}, we introduce the following notation. Let $i\leq\ell\leq n+1$ be positive integers, $C$ an almost Boolean antichain in $\cT_\ell$ and $c\in C$. We write
$$
    \pi_i^\ell(c):=q_i\circ q_{i+1}\circ \dots \circ q_{\ell-1}(c).
$$
Furthermore, for $C=\{c_1,\dots,c_k\}$, let
$$
    \pi_i^\ell(C):=\{\pi_i^\ell(c_1),\dots,\pi_i^\ell(c_k)\}\setminus \{\hat{1}_{\cT_i}\}.
$$
To ease notation, we denote $\pi_i^{n+1}(c)$ by $\pi_i(c)$ and $\pi_i^{n+1}(C)$ by $\pi_i(C)$. We observe that since $C$ is an almost Boolean antichain, it is of the form $C=\{f_{\ell-1}(c_1),\dots, f_{\ell-1}(c_{k-1}),w\}$ where either $\{c_1,\dots,c_{k-1}\}$ or $\{c_1,\dots, c_{k-1},q_{\ell-1}(w)\}$ is a Boolean antichain in $\cT_{\ell-1}$. In both cases, this Boolean antichain is equal to $\pi_{\ell-1}^\ell(C)$. \Cref{prop: BooleanAreValid} shows that $\pi_{\ell-1}^\ell(C)$ is an almost Boolean antichain. Applying the same argument, $\pi_{\ell-2}^\ell(C)$ is a Boolean antichain again. This process can be continued to see that 
\begin{equation}\label{eq: pi of C Boolean}
    \pi_i^\ell(C) \text{ is a Boolean antichain for all } i\leq \ell.
\end{equation}

Moreover, \Cref{cor: k_w = n} yields that
\begin{equation}\label{eq: I_q_i}
    \text{for } d\in \cT_\ell \text{ we have }I_{f_\ell(d)} = I_d \text{ and } I_{q_\ell(d)}\subseteq I_d
\end{equation}
with strict inclusion if and only if $d\neq \text{top}(d)$.

This process allows us to relate Boolean antichains in $\cT_\ell$ to Boolean antichains in $\cT_i$ for $i<\ell$. In the following lemma, we project an element of a Boolean antichain onto an element in a smaller Tamari lattice with a chosen index in its closing index set.

\begin{lemma}\label{lem: backtracking}
    Let $C$ be an almost Boolean antichain in $T_\ell$. For a fixed $c \in C$ and a fixed $j \in I_c$ there exists an integer $m \leq \ell$ such that
    $\pi_m^\ell(c) = w(\pi_m^\ell(C))$ and $j \in [k_{\pi_m^\ell(c)} + 1, m - 1] \subseteq I_{\pi_m^\ell(c)}$.
\end{lemma}
\begin{proof}
    For all $\iota\leq \ell$, we denote $\pi_\iota^\ell(c)$ by $c_\iota$. Now fix $i\leq \ell$. If $c_i = \topmap(c_i)$, then \eqref{eq: I_q_i} shows that $I_{c_{i-1}} = I_{c_i}$. If $c_i \neq \topmap(c_i)$, then let $y = \topmap(c_i)$ be the top element of the equivalence class. We know $I_{c_{i-1}} = I_y$ and there exists a chain $c_i =: y_1 \lessdot \dots \lessdot y_t := y$. By \Cref{lem: closing parentheses in equivalence class}, we have $I_{y_{s-1}} = I_{y_s} \cup \{k_{y_s} \}$. Since $k_{y_{s-1}}$ is the biggest index not included in $I_{y_{s-1}}$, we also have $k_{y_s} > k_{y_{s-1}}$. In total, this results in
    \begin{equation}\label{eq: index set change}
        I_{c_{i-1}} = I_{c_i} \setminus \{k_{y_2} , \dots, k_{y_t} \}, \quad \text{with} \quad k_{y_t} > \dots > k_{y_2} > k_{c_i}.
    \end{equation}
    Therefore, if $j\in I_{c_i}\setminus I_{c_{i-1}}$, then $c_i \neq \topmap(c_i)$ and $j \in [k_{c_i}+1, i-1]$. In that case we have $c_i = w(\pi_i^\ell(C))$.
    By applying $q$ successively, we eventually get $c_i = \hat{1}_{\cT_i}$. But then the index $j$ does not lie in $I_{c_i} = \varnothing$, so we can set $m:= \max\{i\colon j\notin I_{c_i}\}$.
\end{proof}

\begin{notation*}
    In the setting of \ref{lem: backtracking}, we denote the maximal integer $m \leq \ell$ that satisfies the assertion of the lemma by $n(c,j)$.
\end{notation*}

The following lemma allows for better control over the index $k_{\pi_{n(c,j)}^\ell(c)}$.

\begin{lemma}\label{lem: forwarding k}
    Let $C$ be an almost Boolean antichain in $T_\ell$, $c \in C$ and $j \in I_c$. Then we have
    \[
        k_{\pi_{n(c,j)}^\ell(c)} \notin I_c.
    \]
\end{lemma}
\begin{proof}
    Let $w := \pi_{n(c,j)}^\ell(c)$ and $c_i := \pi_i^\ell(c)$ for all $n(c,j) \leq i \leq \ell$. We assume $k_w \in I_c$. By definition, $k_w$ does not lie in $I_w$ so there exists an $i > n(c,j)$ with $k_w \in I_{c_i} \setminus I_{c_{i-1}}$. Since $k_w$ vanishes from $I_{c_i}$, we know from \eqref{eq: I_q_i} that $c_i \neq \topmap(c_i)$, and from \eqref{eq: index set change} that $k_w \in [k_{c_i} + 1, i-1]$. The definition of $n(c,j)$ gives $k_w < j < n(c,j)$. Therefore, we have $j \in [k_{c_i} + 1, i - 1]$. This is a contradiction to the maximality of $n(c,j)$.
\end{proof}

With this we can now complete the recursive description of Boolean antichains.

\begin{proposition}\label{prop: ValidAreBoolean}
    Let $C=\{f_n(c_1),\dots,f_n(c_{k-1}),w\}$ be an almost Boolean antichain of size $k\leq n$ in $\cT_{n+1}$. Suppose that there exists at most one $i \in [1,k-1]$ with $k_w \in I_{f_n(c_i)}$. Then $C$ is a Boolean antichain.
\end{proposition}
\begin{proof}
    Let $S$ and $S'$ be subsets of $ C$. As $C$ is an almost Boolean antichain, \eqref{eq: intersectivity} is fulfilled for $\{c_1, \dots, c_{k-1}, q_n(w)\}$. Since $q$ is a lattice congruence (see \Cref{prop: LatticeCongruence}), the equality $[(\wedge S) \vee (\wedge S')] = [\wedge (S\cap S')]$ holds. By definition, we have $f_n(c_i) = \text{top}(f_n(c_i))$. If $w$ is in at most one of $S, S'$, one can use \Cref{lem:banana} and \Cref{cor: bot and top meet and join} to show that $(\wedge S) \vee (\wedge S')$ and $\wedge (S\cap S')$ have the same height. To do so in general, we compute $k_{\wedge (S\cap S')}$ and argue that $k_{\wedge (S\cap S')}\not\in I_{(\wedge S) \vee (\wedge S')}$, in order to apply \Cref{lem: closing parentheses in equivalence class}. The proof is divided into four steps.

    \noindent{\sf Step 1.}
    We first construct a sequence $(a_i)_{1\leq i\leq r}$ of elements in $C$ and a sequence $(d_i)_{0\leq i\leq r}$ of corresponding elements in different $\pi_{j_i}(C)$. To do this, we iteratively use the process described in \Cref{lem: backtracking} to project the latest $d_i$ onto a smaller Tamari lattice $\cT_m$, where it is a non-top element of its equivalence class. 
    If $i\geq 0$ we denote the index $k_{\pi_m(a_i)}$ by $\overline{k_{d_i}}$ where $m\leq n +1$ will be clear from the context. For $i = -1$, we set $\overline{k_{d_{-1}}} = k_{w}+1$.
    
    First, let $d_0$ be $w$. By the assumption on $C$, there exists at most one element $d_1 \in C$ with $k_w \in I_{d_1}$ and we define $a_1:=d_1$.
    Now suppose we have $a_1, \dots, a_{i}$ and $d_0, \dots, d_{i}$ satisfying $\overline{k_{d_{i-1}}}\in I_{d_{i}}$, and $\pi_{n(d_{i-1}, \overline{k_{d_{i-2}}})}(a_{i}) = d_{i}$.
    We construct $d_{i+1}$ and $a_{i+1}$. By \Cref{lem: backtracking}, there exists $n(d_i,\overline{k_{d_{i-1}}}) < n(d_{i-1},\overline{k_{d_{i-2}}})$ such that \[\pi_{n(d_i, \overline{k_{d_{i-1}}})}^{n(d_{i-1}, \overline{k_{d_{i-2}}})}(d_i) = w(\pi_{n(d_{i}, \overline{k_{d_{i-1}}})}(C)).\]
    By \eqref{eq: pi of C Boolean}, $\pi_{n(d_i,\overline{k_{d_{i-1}}})}(C)$ is a Boolean antichain. Hence, by \Cref{prop: BooleanAreValid}, there exists at most one $d_{i+1} \in \pi_{n(d_i,\overline{k_{d_{i-1}}})}(C)$ with $\overline{k_{d_i}} \in I_{d_{i+1}}$. If it exists, then there exists $a_{i+1} \in C$ with $\pi_{n(d_i, \overline{k_{d_{i-1}}})}(a_{i+1}) = d_{i+1}$. This $a_{i+1}$ is unique: If there exists another $a \in C$ such that $\pi_{n(d_i, \overline{k_{d_{i-1}}})}(a) = d_{i+1}$, then the two elements $\pi_{n(d_i, \overline{k_{d_{i-1}}})+1}(a_{i+1})$ and $\pi_{n(d_i, \overline{k_{d_{i-1}}})+1}(a)$ are in the same equivalence class. However, this gives two comparable elements in an antichain, so they must be equal. Applying this argument repeatedly yields $a = a_{i+1}$.
    
    We stop this process as soon as either there exists no $d_{r+1}$ such that $\overline{k_{d_r}} \in I_{d_{r+1}}$ holds, or $a_{r+1} \notin S \cap S'$. Since $n(d_i,\overline{k_{d_{i-1}}})$ strictly decreases in every iteration, one of these two cases occurs after a finite number of iterations.  
    Moreover, we have by \Cref{lem: backtracking} that
    \begin{equation}\label{eq: k_d_j}
        \overline{k_{d_j}} \in [\overline{k_{d_{j+1}}}+1,n(d_{j+1},\overline{k_{d_j}})-1] \text{ for all } 0 \leq j \leq r-1.
    \end{equation}

    \noindent{\sf Step 2.} We now show that for each $i$ there exists at most one $a \in C$ with $\overline{k_{d_i}} \in I_a$ and that, if it exists, then $i\leq r$, $d_{i+1}$ exists and $a$ must be $a_{i+1}$. In particular, this implies $\overline{k_{d_i}} \in I_{a_{i+1}}$.

    We proceed by induction on $i$. By the assumption of the theorem, the initial step holds, in which $a_1=d_1$ is in $C$. For the induction step, let $a \in C$ such that $\overline{k_{d_i}} \in I_a$. By \Cref{lem: backtracking}, there exists $n(a,\overline{k_{d_i}})$ such that
    \begin{equation}\label{eq: I_a_overline}
        \overline{k_{d_i}} \in [\overline{k_a}+1,n(a,\overline{k_{d_i}})-1] \subseteq I_{\pi_{n(a,\overline{k_{d_i}})}(a)}.    
    \end{equation}
    Suppose that $n(a,\overline{k_{d_i}}) > n(d_i,\overline{k_{d_{i-1}}})$, then \eqref{eq: k_d_j} and \eqref{eq: I_a_overline} yield that
    $$
        \overline{k_a}+1 \leq \overline{k_{d_i}} < \overline{k_{d_{i-1}}} < n(d_i,\overline{k_{d_{i-1}}}) \leq n(a,\overline{k_{d_i}})-1.
    $$
    We use \eqref{eq: I_a_overline} and \eqref{eq: I_q_i} to see that $\overline{k_{d_{i-1}}} \in I_a$. Since $\overline{k_{d_i}} \notin I_{a_i}$ holds by \Cref{lem: forwarding k}, but $\overline{k_{d_i}} \in I_a$ by assumption, $a$ and $a_i$ must be distinct. But by the induction hypothesis, $a_i$ is the only element of $C$ with $\overline{k_{d_{i-1}}}$ in its closing index set, yielding a contradiction. This shows that $n(a,\overline{k_{d_i}}) \leq n(d_i,\overline{k_{d_{i-1}}})$. Let $a'$ in $\cT_{n(d_i,\overline{k_{d_{i-1}}})}$ be such that $a' = \pi_{n(d_i,\overline{k_{d_{i-1}}})}(a)$. Then \eqref{eq: I_q_i} and \eqref{eq: I_a_overline} show that the integer $\overline{k_{d_i}}$ lies in $I_{a'}$. Thus, there exists an element in $\cT_{n(d_i,\overline{k_{d_{i-1}}})}$ with $\overline{k_{d_i}}$ in its index set. But since there exists at most one such element, namely $d_{i+1}$, we have $a'=d_{i+1}$. By the same argument as in {\sf Step 1}, $a$ is equal to $a_{i+1}$.

    \noindent{\sf Step 3.} We now compute $k_{\wedge S\cap S'}$.
    Let $(a_i)_{1\leq i\leq r}$ and $(d_i)_{0 \leq i\leq r}$ be the sequences constructed in {\sf Step 1}. By {\sf Step 2}, if there exists no $d_{r+1}$, then there also exists no $a \in C$ with $\overline{k_{d_r}} \in I_a$. On the other hand, if $d_{r+1}$ exists and $a_{r+1}\notin S\cap S'$, then $a_{r+1}$ is the only element of $C$ such that $\overline{k_{d_r}} \in I_{a_{r+1}}$. In both cases, \Cref{lem: meet height is union} can be applied to obtain $\overline{k_{d_r}} \notin I_{\wedge(S\cap S')}$. Recall that
    $$
    a_r,a_{r-1},\dots,a_1,w \in S\cap S'.
    $$
    Hence, \Cref{lem: meet height is union}, \Cref{lem: backtracking} and \eqref{eq: I_q_i} yield that
    \begin{align*}
        &[\overline{k_{d_r}}+1,n(d_r,\overline{k_{d_{r-1}}})-1] \cup [\overline{k_{d_{r-1}}}+1,n(d_{r-1},\overline{k_{d_{r-2}}})-1]\; \cup \\ &\dots \cup [\overline{k_{d_1}}+1,n(d_1,k_w)-1] \cup 
    [k_w+1,n]\subseteq I_{\wedge(S\cap S')}.
    \end{align*}
    Together with \eqref{eq: k_d_j}, this gives $j \in I_{\wedge(S\cap S')}$ for all $\overline{k_{d_r}}<j\leq n$. Hence, it follows that 
    $$
        k_{\wedge (S \cap S')} = \overline{k_{d_r}}.
    $$
    
    \noindent{\sf Step 4.} As in {\sf Step 3}, we observe that either no element of $C$ contains $\overline{k_{d_{r}}}$ in its closing index set or $a_{r+1}$ exists and $a_{r+1}\notin S\cap S'$. In the latter case, we can assume that $a_{r+1} \notin S$. Hence, no element in $S$ has $\overline{k_{d_r}}=k_{\wedge (S \cap S')}$ in its closing index set. By \Cref{lem: meet height is union}, the closing index set $I_{\wedge S}$ also does not contain $k_{\wedge (S \cap S')}$. Moreover, by \Cref{cor: closing parenthesis moved to the right}, we have 
    \[I_{(\wedge S)\vee (\wedge S')} \subseteq I_{\wedge S},\]
    so $k_{\wedge (S \cap S')}$ is not in $I_{(\wedge S)\vee (\wedge S')}$. Recall that $(\wedge S)\vee (\wedge S') \leq \wedge(S\cap S')$ holds and that $(\wedge S)\vee (\wedge S')$ and $ \wedge(S\cap S')$ lie in the same equivalence class. By \Cref{lem: closing parentheses in equivalence class}, any element lying strictly below $\wedge (S\cap S')$ contains $k_{\wedge(S\cap S')}$ in its closing index set. Hence $\wedge (S\cap S')$ and $(\wedge S) \vee (\wedge S')$ are equal, showing that $C$ satisfies \eqref{eq: intersectivity}. Thus, $C$ is a Boolean antichain. 
\end{proof}

Using \Cref{prop: BooleanAreValid}, we determine a recursive upper bound for the number of Boolean antichains by counting the almost Boolean antichains of size $k$ in $\cT_{n+1}$.

\begin{corollary}\label{lem: UpperBoundsNumberBooleanAntichains}
We have
\begin{equation*}
    |\cC_k(\cT_{n+1})| \leq \sum_{\{c_1,\dots,c_k\}\in \cC_k(\cT_n)} \Bigl( |[f_n(c_1)]| + \dots + |[f_n(c_k)]|\Bigr) - (k-1) |\cC_k(\cT_n)| + n \cdot |\cC_{k-1}(\cT_n)|.
\end{equation*}

\end{corollary}

\begin{proof}
    Using \Cref{prop: BooleanAreValid}, we bound the cardinality of $\cC_k(\cT_{n+1})$ by the number of almost Boolean antichains.
    The rightmost term in the formula counts the almost Boolean antichains of size $k$ in $\cT_{n+1}$ that contain an element that belongs to $[\hat{1}_{\cT_{n+1}}]$. The remaining terms count the almost Boolean antichains of size $k$ in $\cT_{n+1}$ that project to a Boolean antichain of size $k$ in $\cT_{n}$. At most one element of these antichains is not the top of its equivalence class. This yields the leftmost term of the upper bound. Boolean antichains in which all elements are top elements are counted $k$ times. Therefore, we must subtract $(k-1) \cdot |\cC_k(\cT_n)|$.
\end{proof}

Equality holds for the cases $k \in \{0,1,2\}$. In all other cases, it can be verified that this bound is not tight. This is because the set of Boolean antichains in $\cT_{n+1}$ is a proper subset of the set of almost Boolean antichains in $\cT_{n+1}$ (see \Cref{fig:almostBoolean}).

\subsection{Boolean antichains of maximal size}\label{sec: Maximal size}
In this section, we count the Boolean antichains of maximal size in $\cT_{n+1}$. We showed in \Cref{lem:size bounds} that the size of Boolean antichains is bounded by the number of elements covered by $\hat{1}$ in $\cT_{n+1}$. Here, we notice that the set $\cov(\hat{1})$ has size $n$. To count the Boolean antichains of size $n$, we introduce a recursive labelling of their elements.

\begin{definition}
Let $((x_1x_2)x_3) \in \cT_2$ be labelled with the label $(0,0)$. For $k\geq 2$, label every non-top element $v$ of $[\hat{1}_{\cT_k}]$ with $(k-2,h(v))$. Additionally, if $w$ is a labelled element in $\cT_{k-1}$, we label $f_{k-1}(w)$ with the label of $w$.
\end{definition}

Note that each label is only used once for a fixed $\cT_{n+1}$. Moreover, since $[\hat{1}_{\cT_{n+1}}]\setminus\{\hat{1}_{\cT_{n+1}}\}$ has size $n$ (see \Cref{size of equivalence class}), the second entry of each label is always at most as large as the first entry. \Cref{fig: tamari lattice with labels} illustrates the labelling of $\cT_4$. The labels $(1,0)$ and $(1,1)$ are lifted from the non-top elements of $[\hat{1}_{\cT_3}]$ and the label $(0,0)$ is lifted twice from the non-top element of $[\hat{1}_{\cT_2}]$.

\begin{figure}
    \centering
    \scalebox{0.9}{
    \begin{tikzpicture}
    \node (0) at (8, 13.5) {\footnotesize$(x_1(x_2(x_3(x_4 x_5))))$};
    \node (1) at (8, 12.5) {\footnotesize$(x_1(x_2((x_3x_4)x_5)))$};
    \node (4) at (8, 11.5) {\footnotesize$(x_1((x_2(x_3x_4))x_5))$};
    \node (8) at (8, 10.5) {\footnotesize$((x_1(x_2(x_3x_4)))x_5)$};
    
    \node (2) at (13, 10.5) {\footnotesize$((x_1x_2)(x_3(x_4x_5)))$};
    \node (6) at (13, 9.5) {\footnotesize$((x_1x_2)((x_3x_4)x_5))$};
    \node (10) at (13, 8.5) {\footnotesize$(((x_1x_2)(x_3x_4))x_5)$};
    
    \node (3) at (4, 11) {\footnotesize$(x_1((x_2x_3)(x_4x_5)))$};
    \node (5) at (4, 10) {\footnotesize$(x_1(((x_2x_3)x_4)x_5))$};
    \node (9) at (4, 9) {\footnotesize$((x_1((x_2x_3)x_4))x_5)$};
    
    \node (7) at (0, 8.5) {\footnotesize$((x_1(x_2x_3))(x_4x_5))$};
    \node (12) at (0, 7.5) {\footnotesize$(((x_1(x_2x_3))x_4)x_5)$};

    \node (11) at (6, 6) {\footnotesize$(((x_1x_2)x_3)(x_4x_5))$};
    \node (13) at (6, 5) {\footnotesize$((((x_1x_2)x_3)x_4)x_5)$};

    \draw (0) -- (1);
    \draw (0) -- (2);
    \draw (0) -- (3);
    \draw (1) -- (4);
    \draw (1) -- (6);
    \draw (2) -- (11);
    \draw (2) -- (6);
    \draw (3) -- (5);
    \draw (3) -- (7);
    \draw (4) -- (5);
    \draw (4) -- (8);
    \draw (5) -- (9);
    \draw (6) -- (10);
    \draw (7) -- (11);
    \draw (7) -- (12);
    \draw (8) -- (9);
    \draw (8) -- (10);
    \draw (9) -- (12);
    \draw (10) -- (13);
    \draw (11) -- (13);
    \draw (12) -- (13);

    \draw[rounded corners = 1mm, gray, fill=llgray] (1.4, 8.4) rectangle (2.2, 9);
    \node at (1.8,8.7) {\footnotesize$(1,0)$};
    \draw[rounded corners = 1mm, gray, fill=llgray] (5.4, 10.9) rectangle (6.2, 11.5);
    \node at (5.8,11.2) {\footnotesize$(1,1)$};
    \draw[rounded corners = 1mm, gray, fill=llgray] (9.4, 12.4) rectangle (10.2, 13);
    \node at (9.8,12.7) {\footnotesize$(2,2)$};
    \draw[rounded corners = 1mm, gray, fill=llgray] (9.4, 11.4) rectangle (10.2, 12);
    \node at (9.8,11.7) {\footnotesize$(2,1)$};
    \draw[rounded corners = 1mm, gray, fill=llgray] (9.4, 10.4) rectangle (10.2, 11);
    \node at (9.8,10.7) {\footnotesize$(2,0)$};
    \draw[rounded corners = 1mm, gray, fill=llgray] (14.4, 10.4) rectangle (15.2, 11);
    \node at (14.8,10.7) {\footnotesize$(0,0)$};
    \end{tikzpicture}}
    \caption{The labelled elements in $\cT_4$}
    \label{fig: tamari lattice with labels}
\end{figure}
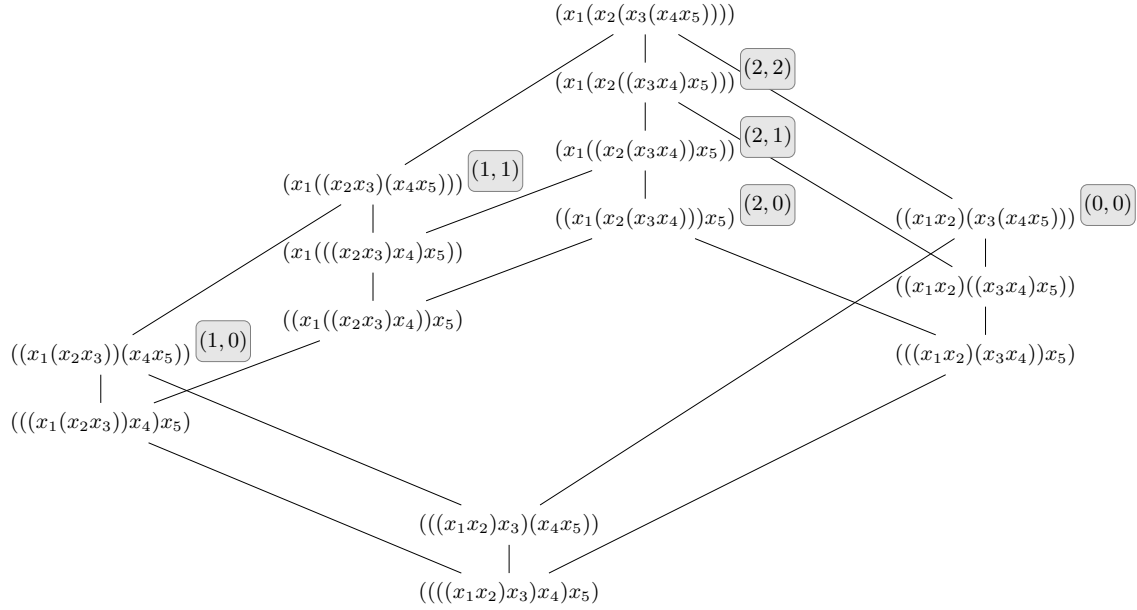

Let $C$ be an almost Boolean antichain of size $n$ in $\cT_{n+1}$. Using \Cref{lem:size bounds}, we see that $C$ consists of $n-1$ elements that are the tops of $q$-equivalence classes corresponding to a Boolean antichain in $\cT_n$, as well as one non-top element of the class $[\hat{1}_{\cT_{n+1}}]$.
In $\cT_2$, the element labelled $(0,0)$ is the only Boolean antichain of size $1$. Thus, we can see that the elements of $C$ have the labels $(0,0),(1,j_1),\dots,(n-1,j_{n-1})$. The first entry of each label is a distinct integer from $0$ to $n-1$. Note that not all antichains with labels of this form are almost Boolean antichains.

\begin{definition}
Let $S$ be a set of labels in $\cT_{n+1}$ of the form \[S = \{(0,0),(1,j_1),\ldots,(n-1,j_{n-1})\}.\]  
Then $S$ is called a \textit{valid labelling} if for every $k \in [1, n-1]$ and for every $i \in [j_k, k-1]$ we have that $j_i \geq j_k$ (see \Cref{fig:triangleValidLabelling}). We say that a set of labels is \textit{compatible} with another label if there exists a valid labelling containing all these labels.
\end{definition}

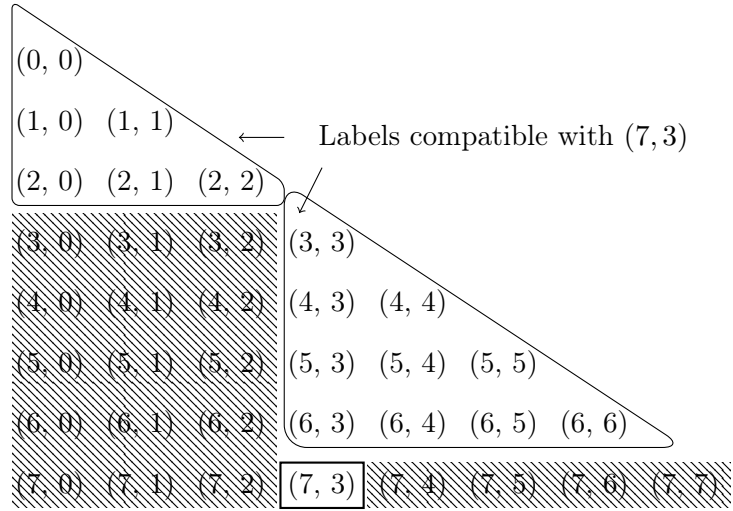
\begin{figure}
    \centering
    \begin{tikzpicture}
         \foreach \x in {0,...,7}
            \foreach \y in {0,...,\x} 
                {\node[]  (\x\y) at (1.2*\y,-0.8*\x) {(\x, \y)};}
    \draw[thick] (3.05, -5.9) rectangle (4.15, -5.3);
    \fill[pattern=north west lines](-0.5, -5.9) rectangle (3, -2);
    \fill[pattern=north west lines](4.2, -5.9) rectangle (9, -5.3);
    \node[](text) at (6, -1) {Labels compatible with $(7, 3)$};
    \draw[rounded corners = 1mm](-0.5, .8)--(3.1, -1.6)--(3.1, -1.9)--(-0.5, -1.9)--cycle;
    \draw[rounded corners = 3mm](3.1, -1.6)--(8.4, -5.1)--(3.1, -5.1)--cycle;
    \draw[<-](2.5, -1)--(3.1, -1);
    \draw[<-](3.3, -2)--(3.6, -1.4);
    \end{tikzpicture}
    \caption{Labels compatible with $(7, 3)$ for a valid labelling in $\cT_9$}
    \label{fig:triangleValidLabelling}
\end{figure}

We now show that valid labellings label Boolean antichains of maximal size.

\begin{lemma}\label{lem: ValidLabellings count BA of max size}
    Let $C$ be an almost Boolean antichain of length $n$ in $\cT_{n+1}$. Then $C$ consists of labelled elements that form a valid labelling if and only if $C$ is a Boolean antichain.
\end{lemma}
\begin{proof}
    Since $C$ is an almost Boolean antichain, it is of the form 
    $$
    C=\{c_0,c_1,\dots,c_{n-1}\},
    $$
    where each element $c_i$ is labelled with the label $(i,j_i)$. The label $(i,j_i)$ originally corresponds to the element $d$ in $[\hat{1}_{\cT_{i+2}}]$ of height $j_i$. Since the closing index set $I_{\hat{1}_{\cT_{i+2}}}$ is empty, we know by \Cref{lem:chain} and \Cref{lem: closing parentheses in equivalence class} that $I_d = [j_i+1, i+1]$. By construction, $c_i$ is of the form
    $$
    c_{i} = f_n \circ f_{n-1} \circ \dots \circ f_{i+2}(d).
    $$
    Using \eqref{eq: I_q_i} yields that $I_{c_i} = [j_i+1, i+1]$ as well. In particular, we have 
    $$
        I_{c_{n-1}} = [j_{n-1}+1, n] \text{ and } k_{c_{n-1}}=j_{n-1}.
    $$
    The inclusion $k_{c_{n-1}} \in I_{c_{j_{n-1}-1}}$ always holds.
    Hence, there exists at most one $c_i \in C$ such that $k_{c_{n-1}} \in I_{c_i}$ if and only if we have $j_i \geq j_{n-1}$ for every $i \in [j_{n-1},n-2]$. 
    
    Using \Cref{prop: ValidAreBoolean}, this shows that if $C$ consists of labelled elements whose labels are a valid labelling, then $C$ is a Boolean antichain.
    Conversely, let $C$ be a Boolean antichain. Then for all positive integers $k < n-1$, we observe that $\pi_{k+2}(C)$ is a Boolean antichain (see \eqref{eq: pi of C Boolean}) consisting of the elements with the labels $(0,0),(1,j_1),\dots,(k,j_k)$. By \Cref{prop: BooleanAreValid} and the discussion above, we see that $j_i \geq j_k$ holds for every $i \in [j_k, k-1]$. This shows that $C$ consists of labelled elements that form a valid labelling.
\end{proof}

With this we are able to count Boolean antichains of maximal size in the Tamari lattice.

\begin{theorem}\label{thm: main_thm_tamari_lattice}
    The number of Boolean antichains of size $n$ in the Tamari lattice $\cT_{n+1}$ is given by the Catalan number $\mathtt{c}_n$.
\end{theorem}
\begin{proof}
    By \Cref{lem: ValidLabellings count BA of max size}, it suffices to enumerate the valid labellings in $\cT_{n+1}$. For this, we partition the labels compatible with a fixed label $(n-1,m)$ into two sets, $S_1$ and $S_2$ (see \Cref{fig:triangleValidLabelling}) and proceed by induction. The lattice $\cT_2$ has one label, hence one valid labelling, and one Boolean antichain. For $n\geq 2$, let
    \[
    S_1 := \{(i,j_i) \; | \; i \in [m, n-2], \, j_i \in [m, i ]\}.
    \]
    By subtracting $m$ from each first and each second entry of each label, we obtain an order preserving bijection
    \begin{align*}
    S_1 &\to S_1' 
    := \{(\ell, j_\ell) \; | \; \ell \in [0, n-2-m], \, j_\ell \in [0, \ell]\} \\
    (i,j) &\mapsto (i-m, j-m).
    \end{align*}
    In particular, a subset $S$ of $S_1$ is compatible with $(n-1,m)$ if and only if its image in $S'_1$ is a valid labelling for $\cT_{n-m}$.
    By induction, there are $\mathtt{c}_{n-m-1}$ valid labellings in $\cT_{n-m}$. Hence, there are $\mathtt{c}_{n-m-1}$ sets of labels in $S_1$ that are compatible with $(n-1,m)$. Now let
    \[
    S_2 := \{(i,j_i) \; | \; i \in [0, m-1], \, j_i \in [0, i]\}.
    \]
    Similarly, there are $\mathtt{c}_m$ sets of labels in $S_2$ that are compatible with $(n-1,m)$. We notice that any pair of labels from $S_1$ and $S_2$ cannot violate the condition for a valid labelling. Since $m$ was chosen arbitrarily from $[0,n-1]$, we can use a well-known identity for Catalan numbers \cite{Segner} to see that there are
    \begin{equation*}
        \sum_{m= 0}^{n-1} \mathtt{c}_m \mathtt{c}_{n-m-1} = \mathtt{c}_n
    \end{equation*}
    valid labellings in total.
\end{proof}

\section{Perspectives}\label{sec: Perspectives}
In this paper, we studied Boolean antichains in specific examples of families of lattices. The success of these computations leads us to formulate several follow-up questions.

\begin{question}
    Can \Cref{prop: BooleanAreValid} and \Cref{prop: ValidAreBoolean} be used to determine a (recursive) formula for the number of Boolean antichains of size $k$ in the Tamari lattice $\cT_{n+1}$ for $2\leq k < n$? Can the construction of Boolean antichains be understood in more general classes of lattices, such as geometric, distributive or modular lattices? Can our usage of lattice congruence be generalized to other settings (as in \Cref{rem: generalisation BooleanAreAlmostBoolean})?
\end{question}

The inclusion of order ideals induces a natural order on Boolean antichains. Based on the formulas explained in this paper, we raise the following questions.

\begin{question}
Is the order on Boolean antichains of maximal size in the Tamari lattice equivalent to the Tamari order, the order on non-crossing partitions, or the order on Dyck paths? What can be said in general about the poset of Boolean antichains in a given lattice?
\end{question}

One can investigate the interplay between the Boolean property of antichains and other properties of the order filters and order ideals generated by them.

\begin{question}
Is it possible to count Boolean antichains whose associated order filters are intervals? 
\end{question}

A less restrictive class of antichains is that of strong antichains, see \cite{gottesman2024antichains} for a definition.

\begin{question}
Can we use our methods to determine the number of strong antichains?
\end{question}

\printbibliography
\end{document}